\documentclass[11pt]{amsart}

\usepackage[margin=1.1in]{geometry}
\usepackage{soul}

\usepackage{amsmath}
\usepackage{amsthm}
\usepackage{amssymb}
\usepackage{colonequals}
\usepackage{tikz}
\usepackage{hyperref}

\newcommand{\ZZ}{\mathbb{Z}}
\newcommand{\QQ}{\mathbb{Q}}
\newcommand{\PP}{\mathbb{P}}

\newcommand{\C}{\mathcal{C}}
\newcommand{\E}{\mathcal{E}}
\newcommand{\F}{\mathcal{F}}
\renewcommand{\O}{\mathcal{O}}

\newcommand{\Pic}{\operatorname{Pic}}
\newcommand{\rk}{\operatorname{rk}}
\newcommand{\codim}{\operatorname{codim}}
\newcommand{\sing}{{\operatorname{sing}}}
\newcommand{\adj}{{\operatorname{adj}}}
\newcommand{\denominator}{\operatorname{denominator}}
\newcommand{\Hom}{\operatorname{Hom}}

\renewcommand{\bar}{\overline}
\renewcommand{\tilde}{\widetilde}
\renewcommand{\hat}{\widehat}
\renewcommand{\phi}{\varphi}
\renewcommand{\setminus}{\smallsetminus}

\newcommand{\andsmall}{\ \text{and} \ }
\newcommand{\andlarge}{\quad \text{and} \quad}
\newcommand{\where}{\qquad \text{where} \qquad}

\newtheorem{thm}{Theorem}[section]
\newtheorem{lm}[thm]{Lemma}
\newtheorem{cor}[thm]{Corollary}
\newtheorem{prop}[thm]{Proposition}

\theoremstyle{definition}
\newtheorem{defi}[thm]{Definition}

\theoremstyle{remark}
\newtheorem{rem}[thm]{Remark}

\def\QQ{{\mathbb Q}}
\def\OO{\mathcal{O}}

\def\F{\mathcal{F}}
\def\P{\mathcal{P}}

\def\E{\mathcal{E}}

\def\I{\mathcal{I}}

\def\cM{\mathcal{M}}

\title{The minimal resolution property for points on general curves}

\author[G. Farkas]{Gavril Farkas}
\address{Farkas: Humboldt-Universit\"at zu Berlin, Institut F\"ur Mathematik,  Unter den Linden 6
\hfill \newline\texttt{}
 \indent 10099 Berlin, Germany} \email{{\tt farkas@math.hu-berlin.de}}

\author[E. Larson]{Eric Larson}

\address{Larson: Department of Mathematics, Brown University, 51 Thayer
Street \hfill
\hfill \newline\texttt{}
 \indent  Providence, RI 02912,  USA}
 \email{{\tt elarson3@gmail.com}}

\begin{document}

\begin{abstract}
We determine when the resolution of a general set of points on a general curve satisfies the Minimal Resolution property. In particular, we completely determine the shape of the minimal resolution of general sets of points on a general curve $C\subseteq \PP^r$ of degree $d\geq 2r$. Our methods also provide a proof (valid in arbitrary characteristic) of the strong version of Butler's conjecture on the stability of syzygy bundles on a general curve of every genus $g>2$ in projective space, as well as of the strong (Frobenius) semistability in positive characteristic of the syzygy bundle of a general curve $C\subseteq \PP^r$ in the range $d\geq 2r$.

\vskip 6pt

\noindent Nous d\'eterminons quand la r\'esolution d'un ensemble g\'en\'eral de points sur une courbe g\'en\'erale
satisfait la propri\'et\'e de r\'esolution minimale. En particulier, nous d\'eterminons compl\`etement la forme de
la r\'esolution minimale d'ensembles g\'en\'eraux de points sur une courbe g\'en\'erale $C$ dans $\PP^r$ de degr\'e $d\geq 2r$.
Nos m\'ethodes fournissent \'egalement une preuve (valable en caract\'eristique arbitraire) de la version forte de la conjecture de Butler
sur la stabilit\'e des fibr\'es de syzygies  sur une courbe g\'en\'erale de genre $g > 2$ quelconque
dans l'espace projectif, ainsi que de la semistabilit\'e forte en caract\'eristique positive du fibr\'e de syzygies d'une courbe g\'en\'erale $C$ dans $\PP^r$ dans l'intervalle $d\geq 2r$.

\end{abstract}

\maketitle

\section{Introduction}

For an embedded projective variety  $X\subseteq \PP^r$ one can ask whether the minimal free resolution of a general set of (sufficiently many) points of $X$ is determined by the geometry of $X$.  We shall provide an essentially complete solution to this question for general curves in projective space.

\vskip 3pt

Setting $S:= k[x_0, \ldots, x_r]$, where $k$ is an algebraically closed field of arbitrary characteristic, we recall that a finitely generated graded $S$-module $M$ has a minimal free resolution
$$0\leftarrow M\leftarrow F_0\leftarrow \cdots \leftarrow F_i\leftarrow \cdots,$$
where $F_i=\bigoplus_{j>0} S(-i-j)^{b_{i,j}(M)}$. The graded Betti numbers $b_{i,j}(M)=\mbox{dim}_k \mbox{Tor}_i^S(M, k)_{i+j}$ are uniquely determined and can be computed via Koszul cohomology. The Betti diagram of $M$ is obtained by placing the entry $b_{i,j}(M)$ in the $i$-th column and $j$-th row.

\vskip 4pt

Let $X\subseteq \PP^r$ be an embedded  projective variety and denote by $P_X(t)$ its Hilbert polynomial. We fix a general subset $\Gamma\subseteq X$ of $\gamma$ points and require that $\gamma\geq P_X(m)$, where $m=\mbox{reg}(X)$ is the Castelnuovo-Mumford regularity of $X$. If $u\geq \mbox{reg}(X)+1$ is the integer determined by the condition $P_X(u-1)\leq \gamma< P_X(u)$, it has been shown in \cite{FMP} that the Betti diagram of $\Gamma$ is obtained from the Betti diagram of $X$ by adding two rows indexed by $u-1$ and $u$, that is, $b_{i,j}(\Gamma)=b_{i,j}(X)$ for $j\leq u-2$, whereas $b_{i,j}(\Gamma)=0$ for $j\geq u+1$. The Minimal  Resolution property  (MRP) for $X$ is the statement
\begin{equation}\label{eq:mrp}
b_{i,u}(\Gamma)\cdot b_{i+1,u-1}(\Gamma)=0,
\end{equation}
for all $\gamma\geq \P_X\bigl(\mathrm{reg}(X)\bigr)$ as described above and for all $i\geq 0$, see \cite{Mus}, \cite{FMP}.
Since the differences $b_{i,u}(\Gamma)-b_{i+1,u}(\Gamma)$ are explicitly determined by the Hilbert polynomial of $X$, the Minimal Resolution property for $X$  determines entirely the Betti diagram of $\Gamma$ and it implies that the Betti numbers of $\Gamma$ are as small as the geometry (that is, the Hilbert polynomial) of $X$ allows.

\vskip 3pt

The Minimal Resolution property (under the name of Minimal Resolution conjecture) has been intensely studied when the variety in question is the projective space. In that case, the resolution of a general set $\Gamma\subseteq \PP^r$ of sufficiently many $\gamma$ points has only two non-trivial rows, indexed $u-1$ and $u$ respectively, and MRP implies that the resolution is \emph{natural}, that is, at each step only one non-trivial Betti number appears. MRP is known to hold for $r\leq 4$, as well as for a very large number of points in any projective space, due to work of Hirschowitz and Simpson \cite{HS}. However, counterexamples to MRP in any projective space $\PP^r$, where $r\geq 6$ and $r\neq 9$, have been found by Eisenbud, Schreyer, Popescu and Walter, see  \cite{EP}, \cite{EPSW}.  The question has also been studied when $X\subseteq \PP^3$ is a smooth surface of small degree, see \cite{BMMN}, or for a $K3$ surface in \cite{AFO}. MRP has been proved to hold for all canonical curves, see \cite{FMP}, and linked to important questions on the moduli space of vector bundles on curves.

\vskip 3pt

We now focus on the case when $X=C$ is a smooth curve embedded by a (not necessarily complete) linear series  $\ell=(L, V)\in G^r_d(C)$. Basic Brill--Noether theory ensures that when $\rho(g,r,d)=g-(r+1)(g-d+r)\geq 0$ the stack $\mathcal{G}^r_d$ parametrizing  such pairs $(C,\ell)$ has a unique component dominating the moduli space $\cM_g$. A pair $[C,\ell]$ corresponding to a (general) point of this component is referred to as a \emph{(general) Brill--Noether (BN)} curve. It was pointed out in \cite{FMP} via vector bundle techniques that property (\ref{eq:mrp}) fails for every curve $C\subseteq \PP^r$  for certain values of $i$ when $d$ is large with respect to $g$. Common to these counterexamples is that they occur in the range $d<2r$ (see also \eqref{eq:bound_mrc} for further explanations). Confirming the expectation, already formulated in \cite{AFO}, that MRP holds outside this range is the main result of this paper.

\begin{thm}\label{thm:main}
Let $C\subseteq \PP^r$ be a  general Brill--Noether curve
of genus $g \geq 1$ and degree $d\geq 2r$. Then the Minimal Resolution property holds for $C$.
\end{thm}

To spell out the statement of Theorem \ref{thm:main}, if $C\subseteq \PP^r$ is a general Brill--Noether curve of degree $d\geq 2r$ and $\Gamma\subseteq C$ is a  general set of $\gamma\geq d\cdot\mbox{reg}(C)+1-g$ points, setting
$$u:=1+\Bigl \lfloor \frac{\gamma+g-1}{d}\Bigr\rfloor,$$
the Betti diagram of $\Gamma$ is obtained by adding to the Betti diagram of $C$ precisely the rows indexed by $u-1$ and $u$ respectively. The entries in these rows are explicitly given as follows:
$$
b_{i,u}(\Gamma)=0 \  \mbox{ for } \ i\leq r\left(1-\Bigl\{\frac{\gamma+g-1}{d}\Bigr\}\right) \ \ \ \ \mbox{ and }
$$
$$
b_{i,u}(\Gamma)=d{r\choose i}\left(\frac{i}{r}+\Bigl\{\frac{\gamma+g-1}{d}\Bigr\}-1\right) \ \ \ \mbox{ for }
\ i>r\left(1-\Bigl\{\frac{\gamma+g-1}{d}\Bigr\}\right).$$
Here $\{x\}=x-\lfloor x\rfloor$ denotes the fractional part of a number $x$.

\vskip 3pt

\begin{table}[htp!]
\begin{center}
\begin{tabular}{|c|c|c|c|c|c|}
\hline
$1$ & $\ldots$ & $i$ & $i+1$   & $\ldots$  \\
\hline
$b_{1,1}(C)$ & $\ldots$ & $b_{i,1}(C)$ & $b_{i+1,1}(C)$ &  $\ldots$ \\
\hline
$\ldots$ & \ldots & $\ldots$ & \ldots &  $\ldots$  \\
\hline
$b_{1,u-2}(C)$ &  $\ldots$  & $b_{i,u-2}(C)$ & $b_{i+1,u-2}(C)$ &   $\ldots$  \\
\hline
$b_{1,u-1}(\Gamma)$ &  $\ldots$  & $b_{i,u-1}(\Gamma)$ & $b_{i+1,u-1}(\Gamma)$ &   $\ldots$  \\
\hline
$b_{1,u}(\Gamma)$ &  $\ldots$  & $b_{i,u}(\Gamma)$ & $b_{i+1,u}(\Gamma)$ &   $\ldots$  \\
\hline
$0$ &  $\ldots$  & $0$ & $0$ &   $\ldots$  \\
\hline
\end{tabular}
\end{center}
    \caption{The Betti table of a general set $\Gamma\subseteq C$ of $\gamma\gg 0$ points.}
    \label{tab:even}
\end{table}

A version of Theorem \ref{thm:main} with a much more restrictive bound for $d$ has been established in \cite{AFO}. In order to clarify the relevance of the condition $d\geq 2r$ to MRP, we recall the Koszul-theoretic interpretation of the Betti numbers of $\Gamma$. If $\ell=(L, V)\in G^r_d(C)$ is the linear system inducing the embedding $C\subseteq \PP^r$, the kernel vector bundle $M_V$ is constructed via the exact sequence
$$0\longrightarrow M_V\longrightarrow V\otimes \OO_C\longrightarrow L\longrightarrow 0.$$ Using standard Koszul cohomology arguments \cite[Proposition 1.6]{FMP}, one finds
\begin{equation}\label{eq:kosz1}
b_{i+1,u-1}(\Gamma)=h^0\Bigl(C, \bigwedge^i M_V\otimes \I_{\Gamma/C}(u)\Bigr) \ \ \mbox{ and } \ b_{i,u}(\Gamma)=h^1\Bigl(C, \bigwedge^i M_V\otimes \I_{\Gamma/C}(u)\Bigr).
\end{equation}
Since $\mbox{rk}(M_V)=r$ and $\mbox{deg}(M_V)=-d$, by Riemann--Roch one computes
$$b_{i+1,u-1}(\Gamma)-b_{i,u}(\Gamma)=\chi\Bigl(C, \bigwedge^i M_V\otimes \I_{\Gamma/X}(u)\Bigr)={r\choose i}\Bigl(-\frac{id}{r}+du-\gamma+1-g\Bigr),$$ which explains how the $u$-th row of the Betti diagram of $\Gamma$ determines its $(u-1)$-st row. Using (\ref{eq:kosz1}) it is easy to show that $C\subseteq \PP^r$ satisfies the Minimal Resolution property if and only if the kernel bundle $M_V$ verifies the following generic vanishing conditions
\begin{equation}\label{eq:raynaud}
H^0\Bigl(C, \bigwedge^i M_V\otimes \xi\Bigr)=0,
\end{equation}
for each $i=0, \ldots, r$ and for a general line bundle
$\xi\in \mbox{Pic}^{g-1+\lfloor{\frac{id}{r}\rfloor}}(C)$, see also \cite[Corollary 1.8]{FMP}. Note that the degree of $\xi$ is chosen maximally in such a way that the vanishing (\ref{eq:raynaud}) could possibly hold, thus the statement (\ref{eq:raynaud}), if true, is sharp.
It turns out that (\ref{eq:raynaud}) is related to a condition introduced by Raynaud \cite{R} and related to the base locus of the (non-abelian) theta linear system on the moduli space of semistable vector bundles on $C$, see  Definition \ref{def:raynaud}. Proving Theorem \ref{thm:main} amounts to constructing for each $d\geq 2r$ a Brill--Noether curve $C\subseteq \PP^r$ of genus $g$ and degree $d$ which verifies the \emph{strong Raynaud condition} (\ref{eq:raynaud}). Note that via the natural identification $T_{\PP^r|C} \cong M_V^{\vee}\otimes L$, the condition
(\ref{eq:raynaud}) can be equally well stated in terms of the restricted tangent bundle $T_{\PP^r|C}$ of the curve.

\vskip 3pt

We now turn to the condition $d\geq 2r$ in the statement of Theorem \ref{thm:main}. Using a  filtration argument due to Lazarsfeld \cite{EL}, \cite{L} further developed in \cite{P}, \cite{Sch}, one can show that if $D_{r-i}$ is a general effective divisor of degree $r-i$ on $C$, one has an injection $\OO_C(D_{r-i})\hookrightarrow \bigwedge^{r-i} M_{V}^{\vee}$. It follows that Raynaud's condition (\ref{eq:raynaud}) implies that $H^0\bigl(C, L^{\vee}\otimes \xi(D_{r-i})\bigr)=0$, for a general line bundle $\xi$ of degree $g-1+\lfloor \frac{id}{r}\rfloor$, that is,
\begin{equation}\label{eq:diffvar}
L^{\vee}\otimes \xi \notin C_{g-1+\lfloor \frac{id}{r}\rfloor -d+r-i}-C_{r-i},
\end{equation}
where the right hand side denotes the corresponding difference variety inside the Jacobian of $C$. In particular, assuming that for a given $0\leq i\leq r$  both inequalities
\begin{equation}\label{eq:inequalities1}
g-1+ \Bigl\lfloor\frac{id}{r} \Bigr\rfloor-d+r-i\geq 0 \ \ \mbox{ and } \ \ g-1+\Bigl\lfloor \frac{id}{r} \Bigr\rfloor-d+r-i+r-i\geq g,
\end{equation}
are satisfied, the difference variety in (\ref{eq:diffvar}) covers the entire Jacobian of $C$ and accordingly (\ref{eq:diffvar}) cannot hold, therefore the statement (\ref{eq:mrp}) fails for every such curve $C\subseteq \PP^r$. It turns out that the inequalities (\ref{eq:inequalities1}) are mutually compatible for some $0\leq i\leq r$ precisely when
\begin{equation}\label{eq:bound_mrc}
(2r-d)g-r\geq 0.
\end{equation}
Therefore (\ref{eq:bound_mrc}) is the range in which MRP definitely fails  for every curve $C\subseteq \PP^r$ of degree $d$ and genus $g$. On the other hand, if $d\geq 2r$ the inequalities in (\ref{eq:inequalities1}) are incompatible and one does not expect counterexamples to MRP, and indeed in Theorem \ref{thm:main} we confirm this expectation.

\vskip 5pt

The proof of the Minimal Resolution property relies on an induction procedure, where the most effort is put into establishing the strongest possible version of Theorem \ref{thm:main} for \emph{elliptic curves}. This statement serves as the base case of the induction argument and is instrumental in constructing  in the range $d\geq 2r$ Brill--Noether curves verifying the Minimal Resolution property. The following statement combines Theorem \ref{thm:gen} and Proposition \ref{prop-e} and states that one can construct elliptic curves of arbitrary degree in projective space, satisfying non-trivial incidence conditions with respect to rational normal curves and  whose restricted tangent bundle is furthermore generic in the sense of Atiyah's classification of vector bundles on elliptic curves.

\begin{thm} \label{thm:gen0}
If $J\subseteq \PP^r$ is a general elliptic curve of degree $d$, let us write \(d = ad_1\) and \(r = ar_1\) with \(\mathrm{gcd}(d_1, r_1) = 1\). Then
$$T_{\PP^r|J} \cong \bigoplus_{i = 1}^a E_i,$$
where \(E_i\) are stable vector bundles of rank \(r_1\) and degree \((r + 1) d_1\),
and
\(\bigl(\mathrm{det}(E_1), \ldots, \mathrm{det}(E_a)\bigr)\) is general in \(\Pic^{(r + 1)d_1}(J) \times \cdots \times \Pic^{(r + 1)d_1}(J)\). Furthermore, in the range $r+1\leq d\leq 2r-1$, and for any \(0 \leq g \leq d + 1\), we may further require that $J$ meets transversally a rational normal curve $R\subseteq \PP^r$ at \(g\) points.
\end{thm}

The proof of Theorem \ref{thm:gen0} relies on degenerating $J$ to a union $J_0\cup L_1\cup \cdots \cup L_{d-r-1}$ of an \emph{elliptic normal} curve $J_0\subseteq \PP^r$ and $1$-secant lines $L_i\subseteq \PP^r$ meeting $J_0$ at a point $p_i$. Furthermore, we judiciously choose a rational normal curve $R\subseteq \PP^r$ meeting $J_0$ at $r+2$ points $n_1, \ldots, n_{r+2}$, as well as the lines $L_i$ at a point $q_i$. This setup is illustrated in the following picture:

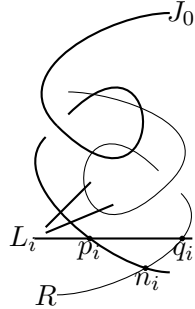
\begin{center}
\begin{tikzpicture}[scale=1.5]
\draw[thick] (1, 2) .. controls (0.5, 2) and (-0.5, 1.5) .. (0, 1);
\draw[thick] (0, 1) .. controls (1, 0) and (1, 2) .. (0.1, 1.1);
\draw[thick] (-0.1, 0.9) .. controls (-0.5, 0.5) and (0.5, -0.3) .. (1, -0.3);
\draw (1.1, 2) node{\(J_0\)};
\draw (-0.1, -0.5) node{\(R\)};
\draw (-0.3, 0) node{\(L_i\)};
\draw (0.1, 1.3) .. controls (0.5, 1.3) and (1.5, 1.0) .. (1, 0.5);
\draw (1, 0.5) .. controls (0, -0.5) and (0, 1.5) .. (0.9, 0.6);
\draw (1.1, 0.4) .. controls (1.5, 0) and (0.5, -0.5) .. (0, -0.5);
\draw[thick] (-0.1, 0.1) -- (0.3, 0.5);
\draw[thick] (-0.1, 0.05) -- (0.5, 0.3);
\draw[thick] (-0.2, 0) -- (1.2, 0);
\filldraw (0.29, 0) circle[radius=0.02];
\filldraw (1.11, 0) circle[radius=0.02];
\draw (0.29, -0.11) node{\(p_i\)};
\draw (1.12, -0.12) node{\(q_i\)};
\filldraw (0.785, -0.263) circle[radius=0.02];
\draw (0.8, -0.38) node{\(n_i\)};
\end{tikzpicture}
\end{center}

The resulting statement for the elliptic normal curve $J_0$ to be proved in order to establish Theorem \ref{thm:gen0} is then a transversality condition for elementary modifications of the restricted tangent bundle $T_{\PP^r|J_0}$. This is  established via a specialization argument inside the moduli space of rational normal curves meeting $J_0$ at the prescribed points $n_1, \ldots, n_{r+2}$ (see Proposition \ref{prop-ss}). We  use throughout a slightly unorthodox stability condition introduced in (\ref{unorth}), which turns out to be particularly suitable when handling vector bundles on families of nodal curves. Less sharp statements similar in spirit to Theorem \ref{thm:gen0} exist in the literature, see \cite{BH}, though for our inductive argument to work we need the result precisely in the form stated in Theorem \ref{thm:gen0}.

\vskip 4pt

The inductive argument in the proof of Theorem \ref{thm:main}  has two parts, as we shall discuss now. One starts with positive integers $g, r$ and $d$ such that $d\geq 2r$ and $\rho(g, r, d)\geq 0$. Assume one has constructed a BN curve $C\stackrel{|V|}\hookrightarrow \PP^r$ of degree $d$ and genus $g$ for which the restricted tangent bundle $T_{\PP^r|C}$ verifies the strong Raynaud condition (\ref{eq:raynaud}). We attach to $C$ a rational normal curve $R\subseteq \PP^r$ meeting $C$ at \(\epsilon + 1\) points for some \(\epsilon \leq r + 1\). The resulting stable curve $C\cup R\subseteq \PP^r$ has degree $d+r$ and (arithmetic) genus $g+r+1$. Note that
$\rho(g + \epsilon, r, d + r) \geq  \rho(g+r+1,r,d+r)=\rho(g,r,d)$ and it is easy to see that $C\cup R$ can be smoothed to a BN curve. Using in an essential way that the restricted tangent bundle $T_{\PP^r|R} \cong \OO_{\PP^1}(r+1)^{r}$ has \emph{integral} slope, we conclude via Lemma \ref{lm:ray-open} that
a smoothing inside $\PP^r$ of $C\cup R$ also satisfies the strong Raynaud condition, therefore establishing the Minimal Resolution property for
the pairs \((d', g') = (d + r, g + \epsilon)\) for \(0 \leq \epsilon \leq r + 1\).

\vskip 4pt

We are thus left with establishing Theorem \ref{thm:main} in the range $2r\leq d\leq 3r-1$. In this case,  Theorem \ref{thm:gen0} yields the existence of a smooth elliptic curve $J\subseteq\PP^r$ such that $T_{\PP^r|J}$ is semistable (which in genus one implies the strong Raynaud condition even in positive characteristic). Furthermore, we can arrange that $J$ meets transversally a rational normal curve $R\subseteq \PP^r$ at $g\leq d-r+1$ points, where this last inequality is a consequence of the constraints imposed on $d$. Using again that the slope of $T_{\PP^r|R}$ is integral, we conclude that a smoothing of $J\cup R$ inside $\PP^r$ is a BN curve of degree $d$ and genus $g$ whose restricted tangent bundle satisfies the strong Raynaud condition. These two inductive steps cover all the cases stated in Theorem \ref{thm:main}.

\vskip 6pt

\noindent {\bf{Butler's conjecture on the stability of kernel bundles.}}

\vskip 4pt

\noindent
The strong Raynaud condition (\ref{eq:raynaud}) necessary to prove the Minimal Resolution property turns out to be stronger than the stability of the kernel bundle $M_V$ of a Brill--Noether general curve $C\subseteq \PP^r$. It has been a long standing conjecture of Butler \cite{Bu} that the kernel bundle $M_V$ is (semi)stable for every $g\geq 3$ and a general choice of $(C, \ell)$. (Note that there is a much studied version of Butler's conjecture for coherent systems of higher rank). Bhosle, Brambilla-Paz and Newstead \cite{BBPN2}, building on significant previous work \cite{AFO}, \cite{BH},  \cite{BBPN1}, \cite{EL}, \cite{Mi} involving a large variety of techniques, managed to show that in characteristic zero the kernel bundle of a BN curve is \emph{semistable} for every $g\geq 1$, and even stable when $g\geq 3$, $r\geq 5$ and $g\geq 2r-4$. Using the degeneration methods of this paper we offer a simple uniform proof of the strongest possible form of Butler's conjecture for general curves in projective space:

\begin{thm}\label{thm:stab}
If $C\subseteq \PP^r$ is a general Brill--Noether curve of genus $g\geq 2$, the kernel bundle $M_V$ is always stable,
unless \(g = 2\) and \(d = 2r\), where $r\geq 3$.
In this case, $M_V$ is  strictly semistable.
\end{thm}

We stress that Theorem \ref{thm:stab} is valid in arbitrary characteristic. This fact and the  \emph{stability} of the bundle $M_V$ in all cases when $g\geq 3$ are new. The strict semistability of the kernel bundle for $g=2$ has been observed before, see \cite[Theorem 8.1]{BBPN1}. If $C\stackrel{|V|}\hookrightarrow \PP^r$ is a genus $2$ curve of degree $2r$ embedded by a linear system $(L,V)$,  then by a dimension count we see that $H^0(C, M_V\otimes \omega_C)\neq 0$,  hence $M_V$ appears as an extension
$$0\longrightarrow \omega_C^{\vee}\longrightarrow M_V\longrightarrow Q\longrightarrow 0,$$
and is therefore strictly semistable. Theorem \ref{thm:stab} shows that it is only this case in genus $2$ when $M_V$ fails to be stable. 
\vskip 6pt

\noindent {\bf{Strong semistability of kernel bundles in positive characteristic.}}

\vskip 4pt

\noindent Turning to the case of a smooth curve $C$  over a field $k$ of characteristic $p>0$, denoting by $F\colon C\rightarrow C$ the absolute Frobenius morphism, it is well known that the pullback under $F$ does not preserve  the stability of vector bundles over $C$.\footnote{A vivid illustration of this fact is provided by the bundle of locally exact differentials defined in Raynaud's paper \cite{R}. The rank $p-1$ vector bundle $B$ on $C$ defined by the exact sequence $0\to B\to F_*\omega_C\to \omega_{C}\to 0$ has been shown to be stable in \cite{R}, but its Frobenius pullback possesses a subbundle $B_2\subseteq F^*B$ such that $F^*B/B_2\cong \omega_C$. Since $\mu(B)=g-1$, this shows that $F^*(B)$ is unstable for $p\neq 2$ and $g\geq 2$.}
Accordingly, a vector bundle $E$ on $C$ is said to be \emph{strongly (semi)stable}
if the Frobenius pullback $(F^e)^*(E)$ is (semi)stable for every $e\geq 0$.

\vskip 3pt

For a smooth  embedded curve $C\stackrel{|V|}\hookrightarrow \PP^r$ of degree $d$ the strong semistability of the syzygy bundle $M_V$ has been related by Brenner \cite{Br} and Trivedi \cite{Tr} to the \emph{Hilbert-Kunz multiplicity} of its coordinate ring. Precisely, for a projective variety $X\subseteq \PP^r$ having coordinate ring $S(X)$, one defines the Hilbert-Kunz function of $X$ by setting
$$\mathbb N\ni e\mapsto \mathrm{HK}_X(p^e):=\mathrm{dim}_k \ \frac{S(X)}{\langle x_0^{p^e}, \ldots, x_r^{p^e}\rangle \cdot S(X)}.$$
Following Kunz and Monsky \cite{Mon}, the  \emph{Hilbert-Kunz multiplicity} $e_{\mathrm{HK}}(X)$ of $X\subseteq \PP^r$ is defined as the leading coefficient of the function $\mathrm{HK}(p^e)$, precisely
\[\mathrm{HK}_X(p^e)=e_{\mathrm{HK}}(X)\cdot p^{e(\mathrm{dim}(X)+1)}+O\bigl(p^{e\cdot \mathrm{dim}(X)}\bigr).\]
In the case when $C\subseteq \PP^r$ is a smooth curve, using Langer's important work \cite{Lan} on strong Harder-Narasimhan filtrations in positive characteristic, it has been showed in \cite{Br} and \cite{Tr} that $e_{\mathrm{HK}}(C, V)$ is a rational number, though its exact value remains hard to compute.  However, it has been observed independently in \cite[Corollary 2.7]{Br}, \cite[Proposition 2.5]{Tr} that under the hypothesis that $M_V$ is strongly semistable the Hilbert-Kunz multiplicity  has the simple formula
$$e_{\mathrm{HK}}(C, V)=\frac{d(r+1)}{2r}.$$
Our techniques yield the following result in this direction:

\begin{thm} \label{thm:semistab-strong}
Let \(C \subseteq \PP^r\) be a very general Brill--Noether curve of genus \(g \geq 1\) and degree \(d \geq 2r\) over a field of characteristic $p>0$.
Then \(M_V\) is strongly semistable, in particular its Hilbert-Kunz multiplicity equals $\frac{d(r+1)}{2r}$.
\end{thm}

The requirement that $C$ be \emph{very} general comes from the fact that the locus of curves $C\subseteq \PP^r$ with a strongly semistable kernel bundle $M_V$ is a countable intersection of open subsets of the stack $\mathcal{G}^r_d$.  Concerning the strong stability of kernel bundles we have the following results:

\begin{thm} \label{thm:stab-strong}
Let \(C \subseteq \PP^r\) be a very general Brill--Noether curve of genus \(g\) and degree \(d\). Then \(M_V\)
is strongly stable in the following cases:
\begin{enumerate}
\item \label{ss1} If \(g \geq r + 1\) and \(d\) is a multiple of \(r\).
\item \label{ss2} If \(g \geq 2\) and \(d \geq 2r + 1\), and \(r\) is not a multiple of the characteristic.
\item \label{ss3} If \(g \geq r + 2\) and \(d \geq 3r\).
\end{enumerate}
\end{thm}

An immediate consequence of Theorem \ref{thm:stab-strong} is that the kernel bundle $M_{\omega_C}$ of a very general canonical curve $C\subseteq \PP^{g-1}$ of genus $g\geq 3$ is strongly stable. Note that whereas we know that the kernel bundle $M_{\omega_C}$ of \emph{every}
non-hyperelliptic canonical curve $C\subseteq \PP^{g-1}$ is stable, we cannot hope for such a result for strong semistability. There are examples of smooth canonical curves $C\subseteq \PP^{g-1}$ of small genus defined over $\mathbb Q$, such that for a Zariski dense set of primes $p$ the mod $p$  reduction of the kernel bundle $M_{\omega_C}$ is strongly semistable and for another dense set of primes the reduction mod $p$ of $M_{\omega_C}$ is not strongly semistable, see Remark \ref{rmk:fermat}.

\vskip 3pt

Key to the proof of both Theorems \ref{thm:semistab-strong} and \ref{thm:stab-strong} is the fact that for elliptic curve in positive characteristics semistability, strong semistability and satisfying the strong Raynaud condition are equivalent properties (see also Lemma \ref{ss-sr}). Using this, the inductive argument used to prove the Minimal Resolution property can be adapted to establish the strong (semi)stability of $M_V$ in every genus.

\vskip 5pt

\noindent  {\small{{\bf{Acknowledgments}}: Farkas was supported by the DFG Grant \emph{Syzygien und Moduli} and by the ERC Advanced Grant SYZYGY of  the European Research Council (ERC) under the European Union Horizon 2020 research and innovation program (grant agreement No. 834172).
Larson was supported by NSF grants DMS-1802908 and DMS-2200641.
Work on this paper has been finalized while the first author visited the R\'enyi Institute of Mathematics in Budapest.
 }}

\section{Syzygies of points on curves and the Raynaud condition}\label{sec:generalities}

We set throughout $S:= k[x_0, \ldots, x_r]$, where $k$ is an algebraically closed field. For a subscheme $\Gamma\subseteq \PP^r$, let $S(\Gamma)$ be its coordinate ring and denote by $$b_{i,j}(Z):=\mbox{dim } \mbox{Tor}_i^S\bigl(S(\Gamma), k\bigr)_{i+j}=\mbox{dim } K_{i,j}\bigl(\Gamma, \OO_{\Gamma}(1)\bigr)$$
the corresponding Betti number of $i$-th syzygies of weight $j>0$. By the very definition of the Castelnuovo--Mumford regularity of $\Gamma$, we have $b_{i,j}(\Gamma)=0$ for $j\geq \mbox{reg}(\Gamma)+1$.

\vskip 4pt

If $\Gamma$ is a subscheme of a smooth curve $C\subseteq \PP^r$, then its Betti numbers can be described geometrically as we now explain.  Let $\ell=(L,V)\in G^r_d(C)$ be a globally generated linear system inducing the map $f=f_{\ell} \colon C \to \PP^r$. We consider the kernel bundle (also referred to as the  Lazarsfeld bundle) $M_V:=\mbox{Ker}\bigl\{V\otimes \OO_C\stackrel{\mathrm{ev}}\longrightarrow L\bigr\}$. Via the Euler sequence, we have the identification  $f^*T_{\PP^r}\cong M_V^{\vee}\otimes L$.
The pullback of the tangent bundle \(f^* T_{\PP^r} = T_{\PP^r|C}\) is intimately related to
many aspects of the geometry of $f$, including its deformation theory when the source curve $C$ is fixed, and
the computation of Koszul cohomology groups of $C$. It is well known that the pair $(C, \ell)$ corresponds to a BN curve if  $H^1(C, f^*T_{\PP^r})=0$.

\vskip 4pt

If $\eta$ is a line bundle on $C$, then the  Koszul cohomology group $K_{i,j}(C;\eta,L)$ is defined as the cohomology of the following complex:
$$\bigwedge^{i+1} H^0(L)\otimes H^0(\eta\otimes L^{j-1})\stackrel{d_{i+1,j-1}}\longrightarrow \bigwedge^i H^0(L)\otimes H^0(\eta\otimes L^j)\stackrel{d_{i,j}}\longrightarrow \bigwedge^{i-1} H^0(L)\otimes H^0(\eta\otimes L^{j+1}),$$
where $d_{i,j}$ denotes the corresponding Koszul differential. Koszul cohomology groups can be describes as ordinary cohomology groups for (twists of) exterior powers of kernel bundles and one has the following well-known identifications:
\begin{equation}\label{eq:koszul_twisted}
K_{i,1}(C; \eta, L)=H^0\Bigl(C, \bigwedge^i M_L\otimes L\otimes \eta\Bigr) \ \ \mbox{ and } \ \ K_{i-1,2}(C;\eta, L)=H^1\Bigl(C, \bigwedge^i M_L\otimes L\otimes \eta\Bigr).
\end{equation}

\vskip 3pt

Assume $\Gamma\subseteq C$ is a subscheme consisting of $\gamma$ distinct points and we make the assumption
\begin{equation}\label{eq:u}
u:=1+\Bigl\lfloor \frac{\gamma+g-1}{d}\Bigr\rfloor\geq 1+\mbox{reg}(C).
\end{equation}
Then, as explained in \cite[Proposition 1.6]{FMP} or \cite[Section 2]{CEFS}, we have
\begin{equation}\label{eq:betti=koszul}
b_{i+1,u-1}(\Gamma)=\mbox{dim } K_{i,1}\bigl(C; L^{u-1}(-\Gamma),L\bigr) \ \ \mbox{ and } \ \ \ b_{i,u}(\Gamma)=\mbox{dim } K_{i-1,2}\bigl(C;L^{u-1}(-\Gamma),L\bigr).
\end{equation}

The Minimal Resolution property (MRP) for the embedded curve $C\stackrel{|V|}\hookrightarrow \PP^r$ is then the statement $b_{i,u}(\Gamma)\cdot b_{i+1,u-1}(\Gamma)=0$ for all $i\geq 0$ and for all integers $u$ satisfying (\ref{eq:u}), that is, the resolution of the $0$-dimensional scheme $\Gamma$ is the expected one and the Betti numbers are as small as the geometry of $C$ allows. We now introduce the following:

\begin{defi} \label{def:raynaud}
We say that a vector bundle \(E\) on a (possibly singular) curve \(C\)
satisfies the \emph{weak Raynaud condition} if, for any degree \(d\),
there exists a line bundle $\eta\in \mbox{Pic}^d(C)$ with either
\begin{equation} \label{h0-or-h1}
H^0(C, E \otimes \eta) = 0 \ \quad \text{or} \ \quad H^1(C, E \otimes \eta) = 0.
\end{equation}
We say that \(E\) satisfies the \emph{strong Raynaud condition} if every wedge power \(\bigwedge^i E\) satisfies
the weak Raynaud condition.
\end{defi}

\begin{rem}
If \(C\) is irreducible, then \(\Pic^d (C)\) is irreducible,
so it is equivalent to ask for \eqref{h0-or-h1} to hold
for a \emph{general} line bundle $\eta\in \mbox{Pic}^d(C)$.
We take existence of $\eta$ in the definition because this behaves better
when \(C\) is reducible and our arguments involve degeneration
to reducible curves.
\end{rem}

Definition \ref{def:raynaud} goes back to the fundamental work of Raynaud \cite{R}. Let $\mathrm{SU}_C(r,d)$ be the moduli space of semistable vector bundles of rank $r$ and fixed determinant of degree $d$ on a smooth curve $C$. It is  known that if $E$ is a stable vector bundle of rank $r$ and integer slope $\mu(E)=\mu\in \mathbb Z$ on $C$,  the point $[E]\in \mathrm{SU}_C(r, r\mu)$ is not a base point of the \emph{determinant line bundle} $\Theta$ generating $\mathrm{Pic}\bigl(\mathrm{SU}_C(r, r\mu)\bigr)$ if and only if $E$ satisfies the weak Raynaud condition, see \cite{P2}.

\vskip 3pt

Via the identifications (\ref{eq:koszul_twisted}) and (\ref{eq:betti=koszul}), we conclude that $C\stackrel{|V|}\hookrightarrow \PP^r$ satisfies the Minimal Resolution property if and only if the kernel bundle $M_V$ satisfies the strong Raynaud condition, which is also precisely condition (\ref{eq:raynaud}) discussed in the introduction.

\subsection{The Raynaud condition under degeneration.}\label{subsec:ray}

In this section, our primary goal is to study the behavior of (semi)stability
and the Raynaud condition under degeneration, which will play a major role in the proof of Theorem \ref{thm:main}.

Let \(\C \to \Delta\) be a family of nodal curves over
the spectrum of a DVR with residue field \(k\) and fraction field \(K\),
and \(\E\) be a vector bundle on \(\C\). We assume the total space $\C$ to be smooth.
Write \(C = \C_{k}\) and \(\C^* = \C_{\bar{K}}\) for the special and general fibers
respectively.
Let \(E := \E|_C\) and \(\E^* := \E|_{\C^*}\).

By inspection, the Raynaud condition is open in families.
In other words, if \(E\) satisfies the weak (respectively strong)
Raynaud condition, then so does \(\E^*\).
The condition that \(E\) satisfies the weak (respectively strong)
Raynaud condition can in turn be expressed as a separate condition on each
irreducible component of \(C\):

\begin{lm} \label{lm:ray-open}
If the restrictions of \(E\) to every component of \(C\)
satisfy the weak (respectively strong) Raynaud condition,
and the slope of \(E\) along all but one component of \(C\) is integral,
then \(E\) satisfies the weak (respectively strong)
Raynaud condition.
\end{lm}
\begin{proof}
Since the conditions of the lemma imply that \(\bigwedge^i E\) also satisfies the conditions
of the lemma, it suffices to consider the case of the weak Raynaud condition. It suffices to show there is a line bundle \(\eta\) of any given degree \(d\) on \(C\)
with \(H^0(C, E \otimes \eta) = 0\) or \(H^1(C, E \otimes \eta) = 0\).
For this we use induction on the number of components
of \(C\); if \(C\) is irreducible, the desired result holds by assumption.

\vskip 3pt

For the inductive step, write \(C = X \cup_\Gamma Y\), where \(Y\) is a component
on which the slope of \(E\) is integral. Let \(\eta\) be a line bundle on \(C\)
such that \(\eta_{|Y}\) is a general line bundle of degree \(g(Y) - 1 - \mu(E_{|Y})\),
and \(\eta_{|X}\) is a line bundle of degree $d - g(Y) + 1+ \mu(E_{|Y})$ on $X$, such that $\eta_{|X}(-\Gamma)\in \mbox{Pic}(X)$ satisfies (\ref{h0-or-h1}) with respect to the vector bundle $E_{|X}$.
Since \(H^0(Y, E_{|Y} \otimes \eta_{|Y}) \cong H^1(Y, E_{|Y} \otimes \eta_{|Y}) = 0\), the long exact sequence
in cohomology attached to the short exact sequence
\[0 \longrightarrow E_{|X} \otimes \eta_{|X}(-\Gamma) \longrightarrow E \otimes \eta \longrightarrow E_{|Y} \otimes \eta_{|Y} \longrightarrow 0\]
implies $H^i(C, E\otimes \eta)\cong H^i(X, E_{|X}\otimes \eta_{|X}(-\Gamma))$ for $i=0,1$, therefore either \(H^0(X, E \otimes \eta) = 0\) or
\(H^1(X, E \otimes \eta) = 0\),
as desired.
\end{proof}

\subsection{Stability conditions for nodal curves.}\label{unorth}

To express that (semi)stability is open,
we first need a good definition of (semi)stability
for vector bundles on nodal curves.
Let \(C\) be a connected nodal curve, and write
\(\nu \colon \tilde{C} \to C\) for the normalization map.
For each node \(p\) of \(C\), let \(p_1\) and \(p_2\)
denote the two points of \(\tilde{C}\) lying over \(C\).
Given a subbundle \(F \subseteq \nu^* E\),
we can compare the subspaces \(F_{|p_1}\) and \(F_{|p_2}\) inside the fibres \(E_{|p_1} \cong E_{|p_2}\).
The following definition, inspired by the concept of stability for parabolic bundles, appears in \cite{clv}.

\begin{defi}\label{def:st1}
Let $E$ be a vector bundle on a connected nodal curve $C$.
For a subbundle \(F \subseteq \nu^* E\) having uniform rank, define the \emph{adjusted slope}
\(\mu^\adj\) via
\[\mu^\adj(F) \colonequals \mu(F) - \frac{1}{\rk(F)} \sum_{p \in C_\sing} \codim_F\bigl(F_{|p_1} \cap F_{|p_2}\bigr).\]
Here, \( \codim_F(F_{|p_1} \cap F_{|p_2})\) refers to the codimension of \(F_{|p_1} \cap F_{|p_2}\)
in either \(F_{|p_1}\) or \(F_{|p_2}\)
(which have the same dimension). Observe that if $F$ is a pull-back of a vector bundle on $C$, then $\mu^\adj(F)=\mu(F)$.

We define \(E\) to be \emph{semistable} if, for all subbundles \(F \subseteq \nu^* E\),
\[\mu^\adj(F) \leq \mu(\nu^* E) = \mu(E),\]
and to be \emph{stable} if this inequality is strict for all proper subbundles of uniform rank.
\end{defi}

Note that Definition \ref{def:st1}  recovers the usual definition of (semi)stability if \(C\) is a smooth curve.
With this definition, (semi)stability is an open condition.

\begin{lm}\label{lm:stab-open}
If $E$ is semistable then so is $\E^*$.
\end{lm}
\begin{proof} This is essentially \cite[Proposition 2.3]{clv}. The proof given in \cite{clv} shows that if \(\E^*\)
has a subbundle \(\F^*\) of (strictly) smaller slope and rank \(s\),
then \(\nu^* E\) has a subbundle of (strictly) smaller adjusted slope
\emph{with the same rank \(s\)}.
\end{proof}
\vskip 3pt

Moreover, the condition that \(E\) is (semi)stable can in turn be expressed
as a separate condition on each irreducible component of \(C\):

\begin{lm}[Lemma 4.1 of \cite{clv}] \label{lm:components-to-special}
If the restrictions of \(E\) to every component
of \(C\) are semistable, then \(E\) is semistable.
If in addition the restriction to one such component is stable,
then \(E\) is stable.
\end{lm}

Combined with \cite[Proposition 2.3]{clv}, we conclude that \(\E^*\) is (semi)stable
in Lemma~\ref{lm:components-to-special}. In positive characteristic, we have a similar result for strong semistability:

\begin{cor} \label{lm:components-to-special-strong}
($\mathrm{char}(k)>0$) If the restrictions of \(E\) to every component
of \(C\) are strongly semistable, then \(E\) is strongly semistable.
If in addition the restriction to one such component is strongly stable,
then \(E\) is strongly stable.
\end{cor}

We now consider a variant of this setup, in the simplest
case where
one component of \(C\) fails to be semistable. In this case, we will relate
the semistability of \(\E^*\) to the semistability of a \emph{modification}
of the restriction of \(E\) to one component, a notion which we now define:

\begin{defi} Let $E$ be a vector bundle on a variety $X$ and \(D \subseteq X\) be a Cartier divisor.
Let \(F\) be a subbundle of the restriction of \(E\) to some subscheme of $X$ containing \(D\).
We define the elementary modification
\[E[D \to F] \colonequals \ker  \Bigl\{E \to \frac{E_{|D}}{F_{|D}}\Bigr\}.\]
\end{defi}

\begin{lm} \label{lm:stab-mod}
Suppose that \(C = Y \cup_p \PP^1\) is a transverse union of a curve $Y$ and $\PP^1$ respectively, meeting at a point $p$, and \(E_{|\PP^1} \cong \OO_{\PP^1}(a - 1)^{\rk E - 1} \oplus \O_{\PP^1}(a)\).
Then \(\E^*\) is (semi)stable provided that the modification
\[E_{|Y}\bigl[p \to \O(a)_{|p}\bigr]\]
is (semi)stable.
\end{lm}
\begin{proof}
One approach, which is straightforward but requires some casework, is to
show that \(E\) is semistable and then apply Lemma~\ref{lm:stab-open}. Alternatively, we  can replace \(\E\) with the
modification \(\E':=\E[\PP^1 \to \O_{\PP^1}(a)]\),
which has the same general fiber. Then
\(\E'_{|D} \cong E_{|D}[p \to \O(a)_{|p}]\),
whereas \(\E'_{|\PP^1} \cong \O_{\PP^1}(a)^{\rk E}\) is semistable.
The desired result follows by combining
Lemmas~\ref{lm:components-to-special} and~\ref{lm:stab-open}.
\end{proof}

\subsection{General vector bundles on rational and elliptic curves}

We shall find it useful to have a concept of a \emph{general} vector bundle on a rational or elliptic curve that takes into account the Birkhoff--Grothendieck and Atiyah classifications for vector bundles on rational and elliptic curves respectively.

\begin{defi}\label{def:gen_vb} (i) We fix positive integers $r$ and $d$ and write $d=ra+b$, with $0\leq b\leq r-1$. A vector bundle $E$ of rank $r$ and degree $d$ on $\PP^1$ is said to be general if $E\cong \OO_{\PP^1}(a)^{r-b}\oplus \OO_{\PP^1}(a+1)^b$.

(ii) We fix a smooth elliptic curve $J$ and positive integers $r$ and $d$, then write $r=ar_1$ and $d=ad_1$, where $a=\mbox{gcd}(r,d)$. A vector bundle $E$ of rank $r$ and degree $d$ on $J$ is said to be general if $E\cong E_1\oplus \cdots\oplus E_a$, where each $E_i$ is a stable vector vector bundle of rank $r_1$ and degree $d_1$ on $J$ such that $\bigl(\mbox{det}(E_1), \ldots, \mbox{det}(E_a)\bigr)$ is a general element of
$\mbox{Pic}^{d_1}(J)\times \cdots \times \mbox{Pic}^{d_1}(J)$.
\end{defi}

For a smooth curve of genus $g\geq 2$ a vector bundle $E$ is said to be general if it corresponds to a general point of the moduli space of stable vector bundles on $C$ of that rank and degree. With this definition in place, keeping the notation above we have the following:

\begin{lm} \label{lm:general} Suppose that \(C = Y \cup_{p_1} \PP^1 \cup_{p_2} \PP^1\cup  \cdots \cup_{p_n} \PP^1\)
is the transverse union of a curve \(Y\) with \(n\) copies of \(\PP^1\), each meeting \(Y\) at a single point \(p_i \in Y\),
with the \(p_i\in \PP^1\) being mutually distinct.
Assume that the restriction of \(E\) to each \(\PP^1\) satisfies
\[E_{|\PP^1} \cong \O_{\PP^1}(a)^{\rk(E)} \ \text{(respectively} \ E_{|\PP^1} \cong \O_{\PP^1}(a - 1)^{\rk(E) - 1} \oplus \O_{\PP^1}(a)\text{)}.\]
If \(E_{|Y}(ap_1 + \cdots + ap_n)\) (respectively \(E_{|Y}\bigl[p_1 \to \O(a)_{|p_1}\bigr]\cdots\bigl[p_n \to \O(a)_{|p_n}\bigr](ap_1 + \cdots + ap_n)\)) is a general vector bundle on $Y$,
then so is \(\E^*\).
\end{lm}
\begin{proof}

As in the proof of Lemma~\ref{lm:stab-mod}, we can reduce the case \(E_{|\PP^1} \cong \O_{\PP^1}(a - 1)^{\rk(E) - 1} \oplus \O_{\PP^1}(a)\)
to the case \(E_{|\PP^1} \cong \O_{\PP^1}(a)^{\rk(E)}\) by replacing \(\E\) with the elementary modification along the divisor
given by the union of all the \(\PP^1\) factors: \(\E[\PP^1 \cup \cdots \cup \PP^1 \to \O(a)]\).
Similarly, we can further reduce to the case \(a = 0\) by replacing \(\E\) with \(\E\otimes \OO_{\mathcal{C}}(a (\PP^1 \cup \cdots \cup \PP^1))\), where with the notation introduced in (\ref{subsec:ray}), we observe that the union of all the $\PP^1$ factors may be viewed as a Cartier divisor on $\C$.
It therefore suffices to consider the case when the restrictions \(E_{|\PP^1} \cong \O_{\PP^1}^{\rk(E)}\) are all trivial.

In this case, \(\E\) is the pullback of a vector bundle $\bar{\E}$ from the surface $\bar{\C}$ obtained from \(\C\) by contracting all of the \(\PP^1\)s,
 which are \((-1)\)-curves. As the central fiber of \(\bar{\E}\) is
\(E_{|Y}\), the result is immediate.
\end{proof}

\subsection{Elementary modification of kernel bundles}

For an embedded curve
we will be concerned primarily with modifications of the restricted tangent bundle
along certain subbundles  which we define as follows:

\begin{defi}
Let $C$ be a nodal curve and $f\colon C \to \PP^r$ a morphism. Set \(\Lambda \subseteq \PP^r\) to be a linear space
not containing the image of any component of \(C\) under \(f\), and with
\(\Lambda \cap f(C_\sing) = \emptyset\). Let
\[T_{f \to \Lambda|C \setminus f^{-1}(\Lambda)} \colonequals f^* T_{\pi_\Lambda} \subseteq f^* T_{\PP^r},\]
where \(\pi_\Lambda\) denotes the map of projection with center \(\Lambda\).
We then define the \emph{pointing bundle} \(T_{f \to \Lambda}\) to be
the unique extension of $T_{f \to \Lambda |C \setminus f^{-1}(\Lambda)}$
to a subbundle of \(f^* T_{\PP^r}\).

For any subscheme \(X \subseteq \PP^r\), we define the following elementary transformation of the restricted tangent bundle
\[f^* T_{\PP^r} [D \to X] \colonequals f^* T_{\PP^r} [D \to T_{f \to \langle X \rangle}],\]
where \(\langle X \rangle \subseteq \PP^r \) denotes the linear span of \(X\).
\end{defi}

The pointing bundles $T_{f\to \Lambda}$ have a transparent interpretation in terms of kernel bundles via a simple secant construction which we now explain. Set $L:=f^*\OO_{\PP^r}(1)$ and $V:=f^*H^0(\PP^r, \OO_{\PP^r}(1))$. Set $D:=f^{-1}(\Lambda)$ viewed as an effective divisor on $C$. Assuming $\mbox{dim}(\Lambda)=a$, write $V'\subseteq V(-D)$ for the $(r-a)$-dimensional subspace such that one has the canonical identification

$$\Lambda\cong \PP\bigl(V/V'\bigr)\cong \PP^{a}.$$

Then the kernel bundle $M_{V'}$ is naturally a subbundle of $M_V$ and the quotient can be identified up to a twist by $L$ with the pointing bundle $T_{f\to \Lambda}$. Precisely, one has the  exact sequence on $C$:

\begin{equation}\label{eq:pointing_syz}
0\longrightarrow T_{f\to \Lambda}\longrightarrow M_V^{\vee}\otimes L \longrightarrow M_{V'}^{\vee}\otimes L\longrightarrow 0.
\end{equation}
Note that $\mbox{rk}(T_{f\to \Lambda})=a+1=\mbox{dim}(\Lambda)+1$. From the sequence (\ref{eq:pointing_syz}), we derive the following:

\begin{lm} Keeping the previous notation, we have $\mathrm{det}(T_{f\to \Lambda})\cong L^{\mathrm{dim}(\Lambda)+1}\bigl(f^{-1}(\Lambda)\bigr)$.
\end{lm}

Similarly, if $X\subseteq \PP^r$ and $\mbox{dim} \langle X\rangle =a$, let $V'\subseteq V$ be the $(r-a)$-dimensional subspace of hyperplanes containing $X$. We then have the following exact sequence involving kernel bundles:

\begin{equation}\label{eq:elem_mod}
0\longrightarrow f^*T_{\PP^r}[D\to X]\longrightarrow M_V^{\vee}\otimes L\longrightarrow M_{V'}^{\vee}\otimes L_{|D}\longrightarrow 0.
\end{equation}

\section{Stability in exact sequences on nodal curves}
The basic strategy of our proof of Theorem \ref{thm:main} will be to degenerate
\(C\) in projective space so that the kernel bundle $M_V$ fits into an exact sequence
with sub and quotient bundles whose slopes are sufficiently close.
In this section, we study this setup in greater generality.
Let
\begin{equation} \label{seq}
0 \longrightarrow S \longrightarrow E \longrightarrow Q \longrightarrow 0
\end{equation}
be a short exact sequence of vector bundles
on a nodal curve \(C\) of genus \(g\). Our goal is to relate semistability
(respectively the weak Raynaud condition) for \(E\), and for \(S\) and \(Q\).

\begin{lm} \label{lm:wr}
Suppose both vector bundles \(S\) and \(Q\) satisfy the weak Raynaud condition, and
\begin{equation} \label{close-slope}
\lceil \mu(S) \rceil \leq \lfloor \mu(Q) \rfloor + 1 \andlarge \lceil \mu(Q) \rceil \leq \lfloor \mu(S) \rfloor + 1.
\end{equation}
Then \(E\) also satisfies the weak Raynaud condition.
\end{lm}
\begin{proof}
Let \(\eta\) be a suitably general line bundle of degree \(d\) on $C$.
Since \(S\) satisfies the weak Raynaud condition,
\(H^0(C, S \otimes \eta) = 0\) if \(d \leq g - 1 - \mu(S)\),
and \(H^1(C, S \otimes \eta) = 0\) if \(d \geq g - 1 - \mu(S)\).
The analogous statement holds for \(Q\).

Since \(d\in \mathbb Z\) and no integer
lies strictly between \(g - 1 - \mu(S)\) and \(g - 1 - \mu(Q)\)
by our assumption \eqref{close-slope},
it follows that \(H^0(C, S \otimes \eta) = H^0(C, Q \otimes \eta) = 0\)
or \(H^1(C, S \otimes \eta) = H^1(C, Q \otimes \eta) = 0\).
Tensoring \eqref{seq} with \(\eta\), we obtain
that either \(H^0(C, E \otimes \eta) = 0\) or \(H^1(C, E \otimes \eta) = 0\).
In other words, \(E\) satisfies the weak Raynaud condition as claimed.
\end{proof}

\begin{lm} \label{ss-from-exact}
Suppose both \(S\) and \(Q\) are semistable, the degree and rank of \(Q\)
are coprime, and
\[\mu(Q) = \min \bigl\{m \in \QQ : m > \mu(E) \andsmall \denominator(m) < \rk(E)\bigr\}.\]
If the sequence \eqref{seq} does not split, then \(E\) is semistable. Otherwise,
\eqref{seq} splits uniquely and \(S\) is the unique destabilizing quotient of \(E\).
\end{lm}
\begin{proof}
Suppose for contradiction, \(F \subseteq E\) is a destabilizing subbundle. Write \(I\) and \(K\) for image and kernel of the
induced map \(F \to Q\).
Since \(Q\) is semistable by assumption, and its degree and rank are coprime, it is stable;
in particular, if \(I \neq 0\), then \(\mu(I) \leq \mu(Q)\) with equality only if \(I = Q\).
Similarly, if \(K \neq 0\), then \(\mu(K) \leq \mu(S) < \mu(Q)\), because $\mu(Q)>\mu(E)$ by assumption.

\vskip 3pt

Since \(\mu(F) \geq \mu(Q)\) by construction, this is only possible if \(I = Q\)
and \(K = 0\). In particular, \(F \cong Q\) is the unique destabilizing
subbundle, and the inclusion
\(F \subseteq E\) gives a splitting.
Such a splitting is unique because \(\Hom(Q, S) = 0\).
\end{proof}

The following pair of results will be used several times when establishing inductively the (semi)stability of kernel bundles under degeneration
in projective space.

\begin{lm} \label{rat-base}
Let $\{\E_b\}_{b\in B}$ be a family of vector bundles on a nodal curve \(C\),
parameterized by a rational variety \(B\).
Assume for two points \(b_1, b_2 \in B\),
the restrictions \(\E_{b_i}\) fit into exact sequences
\[0 \longrightarrow S_i \longrightarrow \E_{b_i} \longrightarrow Q_i \longrightarrow 0\]
satisfying the assumptions of Lemma~\ref{ss-from-exact}. If \(\mathrm{det}(Q_1) \ncong \mathrm{det}(Q_2)\),
then \(\E_b\) is semistable for a general \(b \in B\).
\end{lm}
\begin{proof}
Assume to the contrary that the general vector bundle $\E_b$ is unstable and let
$\F_b$ be a maximal destabilizing subbundle. Since \(B\) is rational, every rational map \(B \dashrightarrow \Pic (C)\)
is constant; thus, $\mbox{det}(\F_b)$ is constant. However, as we specialize to \(b_i\) along any arc, Lemma~\ref{ss-from-exact}
implies that \(\F\) specializes to \(S_i\).
Thus \(\mbox{det}(Q_1)\cong \mbox{det}(Q_2)\), which is a contradiction.
\end{proof}

\vskip 3pt

\begin{lm} \label{rat-base-stab}
Let $\{\E_b\}_{b\in B}$ be a family of vector bundles on a nodal curve \(C\),
parameterized by a rational variety \(B\).
Assume for two points \(b_1, b_2 \in B\),
the restrictions \(\E_{b_i}\) fit into exact sequences
\begin{equation} \label{seq-stab}
0 \longrightarrow S_i \longrightarrow \E_{b_i} \longrightarrow Q_i \longrightarrow 0,
\end{equation}
where \(S_i\) and \(Q_i\) are stable bundles with \(\mu(S_i) = \mu(Q_i) = \mu(\E)\),
which satisfy \(\rk (S_1) = \rk (S_2) = s\) and
\(\rk (Q_1) = \rk (Q_2) = q\). Suppose \(\mathrm{det}(S_1)\ncong \mathrm{det}(S_2)\).
If \(s = q\), suppose furthermore that the sequences \eqref{seq-stab} are nonsplit
and \(\mathrm{det}(S_i) \ncong \mathrm{det}(Q_i)\).
Then \(\E_b\) is stable for a general \(b \in B\).
\end{lm}
\begin{proof}
Write $\mu = \mu(\E)=\mu(\E_b)$, where
 \(b \in B\) is a general point.
Since semistability is open, \(\E_b\) is semistable.
Assume to the contrary that \(\E_b\) is not stable,
that is, there exists a proper subbundle \(F_b \subset \E_b\) satisfying
\(\mu(F_b) = \mu\).
Write \(f = \rk(F_b)\). As we specialize to \(b_i\) along any arc, \(F_b\) limits to a subbundle of
\(\E_{b_i}\) of slope \(\mu\) (cf.\ \cite[Proposition 2.3]{clv}).
Using \eqref{seq-stab}, we see that any such subbundle is isomorphic to
\(S_i\) or \(Q_i\) (with only \(S_i\) permitted if \(s = q\)
by our assumption that \eqref{seq-stab} is nonsplit in this case).
In particular, we have either \(f = s\) or \(f = q\).

\vskip 4pt

Since \(B\) is rational, every rational map \(B \dashrightarrow \Pic(C)\)
is constant, that is, \(\mbox{det}(\E_b)\) is constant.
If \(F_b\) is unique (among subbundles of $\E_b$ with rank \(f\) and slope \(\mu\)), then in particular, \(F_b = \F_{b}\)
would be the restriction of subbundle \(\F \subseteq \E\),
defined on a dense open subset of \(B\).
Using again that any rational map \(B \dashrightarrow \Pic(C)\) is constant, we obtain that \(\mbox{det}(F_b)\) is similarly constant.
However, as we specialize to \(b_i\) along any arc, \(\F\)
would specialize to either \(S_i\) (if \(f = s\)) or \(Q_i\) (if \(f \neq s\)).
Thus \(\mbox{det}(S_1)\cong \mbox{det}(S_2)\), which contradicts our assumption.

\vskip 4pt

It therefore remains to argue that \(F_b\) is unique among all subbundles of $\mathcal{E}_b$ having rank $f$ and slope $\mu$.
Assume to the contrary that \(F_b\) and \(F_b'\) are two proper subbundles of slope \(\mu\) and rank \(f\).
Dualizing if necessary,
we may suppose without loss of generality that \(f = \min\{s, q\}\).
Note that the saturation of
\(F_b + F_b'\subseteq \E_b\) has slope at least (and therefore exactly) \(\mu\). If this saturation is a proper subbundle of \(\E_b\),
then specializing to \(b_i\) along any arc, \(\E_{b_i}\)
has a subbundle of slope \(\mu\) which is not stable.
However, using \eqref{seq-stab}, no such subbundle exists, since the Jordan-H\"older filtration of $\E_{b_i}$ has only two factors.

\vskip 3pt

Otherwise, if the saturation of $F_b+F_b'$ equals \(\E_b\),
then \(2f \geq s + q\), forcing \(f = s = q\).
Moreover,
the natural morphism \(F \oplus F \to \E\) is generically injective, therefore everywhere injective, with cokernel supported
at \(\deg (\E_b) - \deg (F_b \oplus F_b') = 0\) points. Thus, \(\E_b \cong F_b \oplus F_b'\).
Specializing to \(b_i\) along any arc, both \(F_b\) and \(F_b'\) specialize to \(S_i\).
Therefore, by taking limits we find that  $\mbox{det}(S_i)\otimes \mbox{det}(Q_i)\cong \mbox{det}(\E_{b_i})\cong \mbox{det}(S_i)\otimes \mbox{det}(S_i)$,
which contradicts our assumption that \(\mbox{det}(S_i) \neq \mbox{det}(Q_i)\).
\end{proof}

\section{The Weak Raynaud condition and the Ideal Generation property\label{sec:g0}}

In this section we explain an inductive argument
that will be helpful in proving both Theorems~ \ref{thm:main} (Minimal Resolution property) and \ref{thm:stab} (Butler's conjecture).

\begin{lm} \label{lm:project}
Let $\bigl(C, \ell=(L,V)\bigr)$  be a  BN curve of degree \(d < 2r\) and genus $g$. We fix a general point $p\in C$.
If  $M_{V(-p)}$ satisfies the weak Raynaud condition, then so does $M_V$.
\end{lm}

\begin{proof}
Note that if $f\colon C\stackrel{|V|}\longrightarrow \PP^r$ is the map induced by the linear system $\ell$ and $f_p\colon C\rightarrow \PP^{r-1}$ is the projection of $C$ from the point $f(p)$, then  $M_{V(-p)}^{\vee}\cong f_p^*(T_{\PP^{r-1}})\otimes L^{\vee}(p)$.
We use the exact sequence
$$0\longrightarrow M_{V(-p)}\longrightarrow M_V\longrightarrow \OO_C(-p)\longrightarrow 0.$$
Since $-1=\mu(\OO_C(-p))\leq 1+ \lfloor\mu(M_{V(-p)})\rfloor=1+\lfloor-\frac{d-1}{r-1}\rfloor$,
applying Lemma~\ref{lm:wr} we find that $M_V$ satisfies the weak Raynaud condition as desired.
\end{proof}

\begin{rem} Via Lemma~\ref{lm:project}
we recover by induction on \(r\) the well known fact that a rational normal curve $R\subseteq  \PP^r$ (that is, $g=0$ and $d=r$) satisfies the weak Raynaud condition. In fact, using the Birkhoff--Grothendieck classification of vector bundles on $\PP^1$, we obtain that $M_{\OO_R(1)}\cong \OO_{\PP^1}(-1)^{r}$, that is, $R$ satisfies the Minimal Resolution property as well.
\end{rem}

\begin{lm} \label{lm:d-r}
Let $C\subseteq \PP^r$ be a $BN$ curve of genus $g$ and degree $d$. Suppose that \(T_{\PP^r|C}\)
is either stable, semistable, strongly stable, strongly semistable, satisfies the weak (or the strong) Raynaud condition.
Then for any \(0 \leq \epsilon \leq r + 1\),
the same condition holds for a BN curve
of degree \(d + r\) and genus \(g + \epsilon\) in $\PP^r$.
If \(\epsilon = 0\), the same is true for the condition of \(T_{\PP^r|C}\) being a general vector bundle.
\end{lm}
\begin{proof}
Let \(f \colon C \cup_\Gamma \PP^1 \to \PP^r\) be the union of a
BN curve $C\subseteq \PP^r$ of degree \(d\) and genus \(g\),
and a rational normal curve \(f_{|\PP^1}\colon \PP^1\rightarrow R\subseteq \PP^r\), meeting transversally at a set \(\Gamma\) of
\(\epsilon + 1\) points. Note that the slope of $T_{\PP^r|R}$ is integral, therefore
by Lemma~\ref{lm:ray-open},
or Lemma \ref{lm:stab-open} and Lemma~\ref{lm:components-to-special} or Corollary~\ref{lm:components-to-special-strong},
or Lemma~\ref{lm:general} (in the case $\epsilon=0$),
in order to conclude it suffices to establish that \(f\) corresponds to a BN curve.

\vskip 3pt

From the Gieseker--Petri theorem  \(H^1\bigl(C, T_{\PP^r|C}\bigr) = 0\)
and \(H^1\bigl(\PP^1, f^*_{|\PP^1}T_{\PP^r}(-\Gamma)\bigr) = 0\) since \(f_{|\PP^1}^* T_{\PP^r}\)
satisfies the weak Raynaud condition as discussed.
Using the exact sequence
\[0 \longrightarrow f_{|\PP^1}^* T_{\PP^r}(-\Gamma) \longrightarrow f^* T_{\PP^r} \longrightarrow T_{\PP^r|C} \longrightarrow 0,\]
we conclude \(H^1\bigl(C\cup R, f^* T_{\PP^r}) = 0\), and so \(f\) is a BN curve.
\end{proof}

As a byproduct, we obtain a simple alternative proof of the Ideal Generation conjecture for BN curves
(which has been established with more complicated methods in \cite[Theorem 0.3]{AFO}).

\begin{prop} \label{prop:weak}
Let $C\subseteq \PP^r$ be a general BN curve of genus $g$ and degree $d$, where $\rho(g,r,d)\geq 0$. Then \(T_{\PP^r|C}\) satisfies the weak Raynaud condition.
\end{prop}

\begin{proof}
We argue by induction on \(d\). The base case \(d = 1\) (which forces \(r = 1\) and \(g = 0\))
is clear.
For the inductive step, we apply Lemma~\ref{lm:project} if \(d < 2r\),
and Lemma~\ref{lm:d-r} if \(d \geq 2r\).
\end{proof}

The statement of Proposition \ref{prop:weak} goes under the name of the \emph{Ideal Generation Property} (IGP) which holds for a BN curve, without any restriction on $d$. If $\Gamma\subseteq C$ is a general set of $\gamma\geq d\cdot \mbox{reg}(C)-g+1$ points on $C$, then IGP yields the exact value of the Betti numbers $b_{1,j}(\Gamma)$ and $b_{2,j}(\Gamma)$ of the generators of the ideal $\I_{\Gamma/\PP^r}$, as predicted in the stamement of Theorem \ref{thm:main}.

\section{The Raynaud condition for Elliptic Curves\label{sec:g1}}

In this section, we establish that the kernel bundle of a general Brill--Noether \emph{elliptic} curve is a general bundle in the sense of Definition \ref{def:gen_vb}. This is a key fact, a generalization of which will be needed in the inductive proof of the Minimal Resolution property for general Brill--Noether curves.

\begin{thm} \label{thm:gen}
Let $J\subseteq \PP^r$ be a general Brill--Noether elliptic curve of degree $d$. Then \(f^* T_{\PP^r|J}\) is a general vector bundle on \(J\)
as  the map \(f\colon J\hookrightarrow \PP^r\) varies.
\end{thm}

We turn to the proof of Theorem~\ref{thm:gen}, and of the above-mentioned generalization
(stated as Proposition~\ref{prop-e} near the end of the section).
Our first step is as follows:

\begin{prop} \label{prop-sub-for-ss}
Let \(f\colon J \hookrightarrow \PP^r\) be an elliptic normal curve, and \(n_1, n_2, \ldots, n_a \in J\)
be general points, where \(1 \leq a \leq r - 1\).
Let \(x \in \Lambda \colonequals \langle n_1, n_2, \ldots, n_a \rangle \cong \PP^{a-1}\) be an \emph{arbitrary} point in general linear position
relative to \(n_1, n_2, \ldots, n_a\) and let \(R\subseteq \Lambda\) be a general rational normal curve through \(n_1, n_2, \ldots, n_a\) and $x$.
For some \(0 \leq b \leq a\), we choose general points
\(p_1, p_2, \ldots, p_b \in J\) and \(q_1, q_2, \ldots, q_b \in R\).
Then  \(T' \colonequals T_{f \to \Lambda}[p_1 \to q_1]\cdots[p_b \to q_b]\)
is semistable.
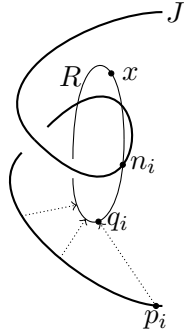
\begin{center}
\begin{tikzpicture}[scale=1.7]
\draw[thick] (1, 2) .. controls (0.5, 2) and (-0.5, 1.5) .. (0, 1);
\draw[thick] (0, 1) .. controls (1, 0) and (1, 2) .. (0.1, 1.1);
\draw[thick] (-0.1, 0.9) .. controls (-0.5, 0.5) and (0.5, -0.3) .. (1, -0.3);
\draw (1.1, 2) node{\(J\)};
\draw (0.3, 0.85) .. controls (0.3, 1.8) and (0.7, 1.8) .. (0.7, 1);
\draw (0.3, 0.7) .. controls (0.3, 0.2) and (0.7, 0.2) .. (0.7, 1);
\filldraw (0.6, 1.52) circle[radius=0.02];
\draw (0.75, 1.52) node{\(x\)};
\filldraw (0.69, 0.805) circle[radius=0.02];
\draw (0.85, 0.805) node{\(n_i\)};
\draw (0.28, 1.5) node{\(R\)};
\filldraw (0.95, -0.3) circle[radius=0.02];
\filldraw (0.5, 0.36) circle[radius=0.02];
\draw[->, densely dotted] (0.95, -0.3) -- (0.51, 0.35);
\draw (0.95, -0.41) node{\(p_i\)};
\draw (0.65, 0.36) node{\(q_i\)};
\draw[->, densely dotted] (0.2, 0.075) -- (0.4, 0.38);
\draw[->, densely dotted] (-0.1, 0.405) -- (0.33, 0.49);
\end{tikzpicture}
\end{center}
\end{prop}
\begin{proof} Note that from the exact sequence (\ref{eq:pointing_syz}) we have that $T_{f\to n_i}\cong \OO_J(1)(n_i)$, hence
\[T_{f \to \Lambda} = \bigoplus_{i = 1}^a T_{f \to n_i} \cong \bigoplus_{i=1}^a \O_J(1)\bigl(n_i\bigr),\]
which is evidently semistable. Thus, for $b=0$ the vector bundle $T'$ is semistable as well.
Similarly, if \(b = a\), then we may specialize each point \(q_i\) to \(n_i\) and then
\(T'\) specializes to
\[T_{f \to \Lambda}[p_1 \to n_1]\cdots[p_a \to n_a] \cong \bigoplus_{i = 1}^a \O_J(1)(n_i - p_1 - \cdots - \hat{p_i} - \cdots - p_a),\]
which is again evidently semistable.
We may thus suppose \(0 < b < a\). Let $z, w\in \mathbb N$ such that
\begin{equation} \label{zwdef}
\frac{z}{w} = \min\left\{\frac{z'}{w'} \in \QQ : \frac{z'}{w'} > \frac{b}{a} \andsmall w' < a\right\}.
\end{equation}

Note that $z\leq b$. The space of $b$-pointed rational normal curves $R\subseteq \Lambda$ passing through \(n_1, n_2, \ldots, n_a, x\) is itself a rational variety.
We can therefore apply Lemma~\ref{rat-base} by constructing certain
degenerations of the pointed curve \((R, q_1, \ldots, q_b)\).
To that end, let
\[\{x'\}:= \Lambda_Q \cap \Lambda_S\subseteq \Lambda, \where \Lambda_Q = \langle x, n_1, n_2, \ldots, n_w \rangle \andlarge \Lambda_S = \langle n_{w + 1}, n_{w + 2}, \ldots, n_a \rangle.\]
We then degenerate \(R\) to a reducible curve \(R_Q \cup R_S\),
where \(R_Q\subseteq  \Lambda_Q\) is a rational normal curve (of degree \(w\)) passing through
\(n_1, n_2, \ldots, n_w, x, x'\) ---
onto which we specialize \(q_1, q_2, \ldots, q_z\) ---
and \(R_S\subseteq \Lambda_S\) is a rational normal curve (of degree \(a - w - 1\)) passing through
\(n_{w + 1}, n_{w + 2}, \ldots, n_a, x'\) ---
onto which we specialize the marked points \(q_{z + 1}, q_{z + 2}, \ldots, q_b\):

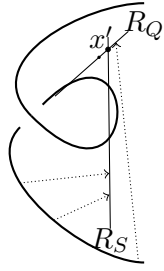
\begin{center}
\begin{tikzpicture}[scale=1.5]
\draw[thick] (1, 2) .. controls (0.5, 2) and (-0.5, 1.5) .. (0, 1);
\draw[thick] (0, 1) .. controls (1, 0) and (1, 2) .. (0.1, 1.1);
\draw[thick] (-0.1, 0.9) .. controls (-0.5, 0.5) and (0.5, -0.3) .. (1, -0.3);
\filldraw (0.6, 1.52) circle[radius=0.01];
\draw[->, densely dotted] (0.95, -0.3) -- (0.75, 1.65);
\draw[->, densely dotted] (0.2, 0.075) -- (0.69, 0.3);
\draw[->, densely dotted] (-0.1, 0.405) -- (0.69, 0.49);
\draw (0.7, 0) -- (0.68, 1.7);
\draw (0.21, 1.175) -- (0.86, 1.75);
\filldraw (0.68, 1.59) circle[radius=0.02];
\draw (0.62, 1.7) node{\(x'\)};
\draw (1.0, 1.8) node{\(R_Q\)};
\draw (0.7, -0.1) node{\(R_S\)};
\end{tikzpicture}
\end{center}

For  a map \(h \colon X \to \PP^r\) or for a subscheme $X\subseteq \PP^r$, write \(\bar{h}
\colon X \to \PP^{r + w - a}\) respectively $\bar{X}\subseteq \PP^{r+w-a}$
for the composition of $h$ with the projection having center \(\Lambda_S\cong \PP^{a-w-1}\). This
projection map induces an exact sequence on $J$
\[0 \longrightarrow T_{f \to \Lambda_S} \longrightarrow T_{f \to \Lambda} \longrightarrow T_{\bar{f} \to
\bar{\Lambda}}\bigl(f^{-1}(\Lambda_S)\bigr) = T_{\bar{f} \to
\bar{\Lambda_Q}}(n_{w + 1} + \cdots + n_a) \longrightarrow 0.\]

Upon specializing the points \(q_i\) to points $q_i^\circ$ as above, the fibers \(T_{f \to
q_i^\circ}|_{p_i}\) are transverse to \(T_{f \to \Lambda_S}|_{p_i}\)
for \(1 \leq i \leq z\), and lie in \(T_{f \to \Lambda_S}|_{p_i}\) for
\(z + 1 \leq i \leq b\). The above exact sequence therefore induces an
exact sequence of modifications:
\begin{multline} \label{exact}
0 \longrightarrow T_{f \to \Lambda_S}\bigl(-p_1 - \cdots - p_z\bigr) [p_{z + 1} \to q_{z +
1}^\circ] \cdots [p_b \to q_b^\circ] \longrightarrow T_{f \to \Lambda}[p_1 \to
q_1^\circ]\cdots[p_b \to q_b^\circ] \\
\longrightarrow T_{\bar{f} \to \bar{\Lambda_Q}}(n_{w + 1} + \cdots + n_a - p_{z +
1} - \cdots - p_b)[\bar{p_1} \to \bar{q_1^\circ}]\cdots[\bar{p_z} \to
\bar{q_z^\circ}] \longrightarrow 0.
\end{multline}

%\begin{multline} \label{exact}
%0 \to T_{f \to \Lambda_S}(-p_1 - \cdots - p_z) [p_{z + 1} \to q_{z + 1}^\circ] \cdots [p_x \to q_x^\circ] \to T_{f \to \Lambda}[p_1 \to %q_1^\circ]\cdots[p_x \to q_x^\circ] \\
%\to T_{\bar{f} \to \bar{\Lambda_Q}}(n_{w + 1} + \cdots + n_y - p_{z + 1} - \cdots - p_x)[\bar{p_1} \to \bar{q_1^\circ}]\cdots[\bar{p_z} \to %\bar{q_z^\circ}] \to 0.
%\end{multline}
By our inductive hypothesis, both the sub and quotient bundle in the above
exact sequence are semistable; by \eqref{zwdef}, this sequence
satisfies the assumptions of Lemma~\ref{ss-from-exact}.
To complete the proof, it suffices by Lemma~\ref{rat-base} to observe that
\begin{multline*}
\mbox{det}\Bigl(T_{\bar{f} \to \bar{\Lambda_Q}}(n_{w + 1} + \cdots + n_a - p_{z + 1} - \cdots - p_b)[\bar{p_1} \to \bar{q_1^\circ}]\cdots[\bar{p_z} \to \bar{q_z^\circ}]\Bigr) \\
\cong \OO_J\bigl(wH + n_1 + \cdots + n_w - (w - 1)p_1 - \cdots - (w - 1)p_z - w p_{z + 1} - \cdots - wp_b\bigr)
\end{multline*}
depends nontrivially on the ordering of the points \(n_i\), as \(0 < w < a\) and \(n_1, n_2, \ldots, n_a\in J\) are general.
\end{proof}

\begin{prop} \label{prop-ss}
Let \(J \subseteq  \PP^r\) be an elliptic normal curve, $n_1, \ldots, n_{r+2}\in J$ be general points and
let \(R\) be a general rational normal curve
meeting \(J\) at \(n_1, n_2, \ldots, n_{r + 2}\).
If \(p_1, \ldots, p_m \in J\) and \(q_1, \ldots, q_m \in R\) are general points where $m\leq r-1$, then
the elementary modification \(T:=T_{\PP^r|J}[p_1 \to q_1] \cdots [p_m \to q_m]\bigl(2p_1+\cdots+2p_m\bigr)\)
is a general vector bundle on $J$.
\begin{center}
\begin{tikzpicture}[scale=1.5]
\draw[thick] (1, 2) .. controls (0.5, 2) and (-0.5, 1.5) .. (0, 1);
\draw[thick] (0, 1) .. controls (1, 0) and (1, 2) .. (0.1, 1.1);
\draw[thick] (-0.1, 0.9) .. controls (-0.5, 0.5) and (0.5, -0.3) .. (1, -0.3);
\draw (1.1, 2) node{\(J\)};
\draw (-0.1, -0.5) node{\(R\)};
\filldraw (0.785, -0.263) circle[radius=0.02];
\draw (0.8, -0.38) node{\(n_i\)};
\draw (0.1, 1.3) .. controls (0.5, 1.3) and (1.5, 1.0) .. (1, 0.5);
\draw (1, 0.5) .. controls (0, -0.5) and (0, 1.5) .. (0.9, 0.6);
\draw (1.1, 0.4) .. controls (1.5, 0) and (0.5, -0.5) .. (0, -0.5);
\draw[->, densely dotted] (0.6, 1.925) -- (0.2, 1.295);
\draw[->, densely dotted] (-0.1, 0.405) -- (0.255, 0.405);
\draw[->, densely dotted] (0.2, 0.072) -- (1.17, 0.2);
\filldraw (0.2, 0.072) circle[radius=0.02];
\filldraw (1.19, 0.2) circle[radius=0.02];
\draw (1.3, 0.2) node{\(q_i\)};
\draw (0.1, 0) node{\(p_i\)};
\end{tikzpicture}
\end{center}
\end{prop}

\begin{proof}
For $f\colon J\hookrightarrow \PP^r$, we argue along the lines of the proof of Proposition~\ref{prop-sub-for-ss}.

If \(m = r - 1\), we may specialize each point \(q_i\) to \(n_i\) for $i=1,\ldots, r - 1$;
under this specialization, denoting by $\pi_{\langle n_1, \ldots, n_{r-1}\rangle} \colon \PP^{r}\dashrightarrow \PP^1$ the projection, \(T\) fits into an exact sequence:
\begin{multline*}
0 \longrightarrow \bigoplus_{i = 1}^{r - 1} T_{f \to n_i}\Bigl(p_i +\sum_{j=1}^{r-1}p_j\Bigr)  \longrightarrow T \\
\longrightarrow \bigl(\pi_{\langle n_1, \ldots, n_{r-1}\rangle} \circ f\bigr)^* T_{\PP^1}\Bigl(\sum_{j=1}^{r-1}(p_j+n_j)\Bigr) \cong \O_J(2)\Bigl(\sum_{j=1}^{r-1}(p_j-n_j)\Bigr) \longrightarrow  0.
\end{multline*}

Since \(T_{f\to n_i}(p_i+p_1+\cdots+p_{r-1})\cong \O_J(1)(n_i + p_i + p_1 + \cdots + p_{r - 1})\), and
\(\O_J(2)\Bigl(\sum_{j=1}^{r-1} (p_j-n_j)\Bigr)\),
are in general non-isomorphic line bundles of the same degree \(2r + 2\), this sequence must split, that is, \(T\)
is the direct sum of these line bundles, and is therefore general.
We may thus assume \(0 < m + 1 < r\).

If $\mbox{gcd}(m+1,r)=1$, then we claim it suffices to show
that \(T\) is semistable. Indeed,
\(T\) is of rank \(r\) and degree \(r(r + 2 - m) + m + 1\),
which are relatively prime; thus, \(T\) would be stable.
Moreover, \(\mbox{det}(T)=\OO_J(r+1)\bigl((r+1)(p_1+\cdots+p_m)\bigr)\) is general.
We conclude that \(T\) would be general if it were semistable, as claimed.
In this case, we define $z, w\in \mathbb N$ such that
\begin{equation} \label{xydef}
\frac{z}{w} = \min\left\{\frac{z'}{w'} \in \QQ : \frac{z'}{w'} > \frac{m + 1}{r} \andsmall w' < r\right\}.
\end{equation}

Otherwise, write \(m + 1 = ka\) and \(r = kb\), with $k=\mbox{gcd}(m+1,r)$
and set \(z: = (k - 1)a\) and \(w: = (k - 1)b\).

\vskip 4pt

Note that the space of $m$-pointed rational curves through \(n_1, n_2, \ldots, n_{r + 2}\) is itself rational.
If \(\mbox{gcd}(m + 1, r) = 1\), we can therefore apply Lemma~\ref{rat-base} by constructing certain
degenerations of \((R, q_1, \ldots, q_m)\), which will also imply the desired result when \(\mbox{gcd}(m + 1, r) \neq 1\).
To construct these degenerations, let
\[\{x'\} = \Lambda_Q \cap \Lambda_S, \where \Lambda_Q = \langle n_1, n_2, \ldots, n_{w + 2}\rangle \andlarge \Lambda_S = \langle n_{w + 3}, \ldots, n_{r + 2}\rangle.\]
We then degenerate \(R\) to a reducible curve \(R_Q \cup R_S\),
where \(R_Q\) is a rational normal curve in \(\Lambda_Q\) of degree \(w + 1\) passing through the points
\(n_1, n_2, \ldots, n_{w + 2}, x'\) ---
onto which we specialize the marked points \(q_1, q_2, \ldots, q_{z - 1}\) ---
and \(R_S\) is a rational normal curve in \(\Lambda_S\) of degree \(r - w - 1\) passing through
\(n_{w + 3}, \ldots, n_{r + 2}, x'\) ---
onto which we specialize the marked points \(q_{z}, \ldots, q_m\).

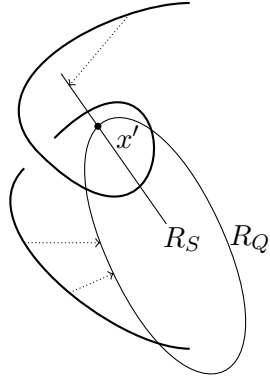
\begin{center}
\begin{tikzpicture}[scale=2]
\draw[thick] (1, 2) .. controls (0.5, 2) and (-0.5, 1.5) .. (0, 1);
\draw[thick] (0, 1) .. controls (1, 0) and (1, 2) .. (0.1, 1.1);
\draw[thick] (-0.1, 0.9) .. controls (-0.5, 0.5) and (0.5, -0.3) .. (1, -0.3);
\draw (0.145, 1.53) -- (0.845, 0.53);
\draw (0.4838, 0.22) .. controls (-0.0362, 1.3328) and (0.67, 1.6628) .. (1.19, 0.55);
\draw (0.4838, 0.22) .. controls (1.0038, -0.8928) and (1.71, -0.5628) .. (1.19, 0.55);
\filldraw (0.39, 1.18) circle[radius=0.02];
\draw (0.59, 1.11) node{\(x'\)};
\draw (0.95, 0.43) node{\(R_S\)};
\draw (1.4, 0.43) node{\(R_Q\)};
\draw[->, densely dotted] (0.6, 1.925) -- (0.2, 1.45);
\draw[->, densely dotted] (0.2, 0.072) -- (0.493, 0.2);
\draw[->, densely dotted] (-0.1, 0.405) -- (0.405, 0.405);
\end{tikzpicture}
\end{center}

Keeping the notation introduced in the proof of Proposition \ref{prop-sub-for-ss}, upon the specialization of the points \(q_i=q_i^\circ\) as described above, setting
$$S:=T_{f \to \Lambda_S}\bigl(p_1 + \cdots + p_{z - 1} + 2 p_z + \cdots + 2p_m\bigr) [p_z \to q_z^\circ] \cdots [p_m \to q_m^\circ]\   \mbox{ and }$$  $$Q:=\bar{f}^* T_{\PP^w}\bigl(n_{w + 3} + \cdots + n_{r + 2} + 2 p_1 + \cdots + 2p_{z - 1} + p_z + \cdots + p_m\bigr)\bigl[\bar{p_1} \to \bar{q_1^\circ}\bigr]\cdots\bigl[\bar{p_{z - 1}} \to \bar{q_{z - 1}^\circ}\bigr],$$ we obtain the following exact sequence on $J$
\begin{equation} \label{exact2}
0 \longrightarrow  S \longrightarrow T \longrightarrow Q \longrightarrow 0.
\end{equation}
By Proposition~\ref{prop-sub-for-ss}, $S$ is semistable, whereas by our inductive hypothesis, \(Q\) is a twist of a general vector bundle, and thus semistable.
Moreover, writing
\[\bar{H} = \bar{f}^* \OO_{\PP^w}(1) =\OO_J(1)\bigl(- n_{w + 3} - \cdots - n_{r + 2}\bigr),\]
we have by direct computation
\begin{align*}
\mbox{det}(S) &= (r - w)(\bar{H} + p_1 + \cdots + p_{z - 1}) + (r - w + 1)(n_{w + 3} + \cdots + n_{r + 2} + p_z + \cdots + p_m), \\
\mbox{det}(Q) &= (w + 1)(\bar{H} + p_1 + \cdots + p_{z - 1}) + w(n_{w + 3} + \cdots + n_{r + 2} + p_z + \cdots  + p_m).
\end{align*}

\vskip 4pt

We distinguish two cases. If \(\mbox{gcd}(r, m + 1) = 1\), via \eqref{xydef}, the sequence \eqref{exact2}
satisfies the assumptions of Lemma~\ref{ss-from-exact}.
To complete the proof, it suffices, by Lemma~\ref{rat-base}, to observe that \(\mbox{det}(Q)\)
depends nontrivially on the ordering of the \(n_i\).

\vskip 3pt
If \(\mbox{gcd}(r, m + 1) \neq 1\), by our inductive hypothesis \(Q\) is a twist of the general vector bundle
\(Q' = \bar{f}^* T_{\PP^w}(2 p_1 + \cdots + 2p_{z-1})[\bar{p_1} \to \bar{q_1^\circ}]\cdots[\bar{p_{z - 1}} \to \bar{q_{z - 1}^\circ}]\).
Since \(S\) is semistable, with rank and degree that are relatively prime,
and \(\mu(S) = \mu(Q)\), it suffices to observe that (i)
\(\mbox{det}(S)\) and \(\mbox{det}(Q)\) are independently general, and (ii) once \(\mbox{det}(S)\) and \(\mbox{det}(Q)\) are fixed,
the parameters \(p_1, q_1^\circ, \ldots, p_{z - 1}, q_{z - 1}^\circ\) and \(\bar{H}\)
that determine \(Q'\) remain general subject only to the \emph{divisorial} constraint
that \(\bar{H} + p_1 + \cdots + p_{z - 1}\) is a fixed general divisor class.
This constraint determines \(\mbox{det}(Q')\), and so
\(Q'\) remains general subject only to the constraint that \(\mbox{det}(Q')\) is fixed. \qedhere

\end{proof}

Consider the union
\(f \colon J'=J_0 \cup L_1 \cup \cdots \cup L_{d-r-1} \to \PP^r\),
of an elliptic normal curve \(J_0\subseteq \PP^r\) with \(1\)-secant lines \(L_i\subseteq \PP^r\) meeting \(J_0\) at \(p_i\), such that \(J_0\) meets \(R\) at \(r + 2\) points
\(n_1, n_2, \ldots, n_{r + 2}\),
and each \(L_i\) meets \(R\) at one point \(q_i\). This setup was illustrated in the picture in the introduction.

\vskip 3pt

\begin{lm} \label{er-is-bn}
We have \(H^1\bigl(J', f^* T_{\PP^r}(-n_1 - \cdots - n_{r + 2} - q_1 - \cdots - q_{d-r-1})\bigr) = 0\).
In particular, we may smooth
any subset of the nodes of \(J'\) while requiring that \(f(J')\)
meet \(R\) exactly in \(\{n_1, n_2, \ldots, n_{r + 2}, q_1, q_2, \ldots, q_{g - r - 2}\}\).
\end{lm}
\begin{proof}
We use the following exact sequence on $J'$
\begin{multline*}
0 \longrightarrow \bigoplus_{i = 1}^{d-r-1}  T_{\PP^r|L_i}(-p_i - q_i) \longrightarrow f^* T_{\PP^r}\Bigl(-\sum_{i=1}^{r+2} n_i - \sum_{j=1}^{d-r-1} q_j\Bigr) \longrightarrow  T_{\PP^r|J_0}\Bigl(-\sum_{i=1}^{r+2} n_i\Bigr) \longrightarrow 0.
\end{multline*}
Note that \(T_{\PP^r|L_i} \cong \O_{L_i}(1)^{r - 1} \oplus \O_{L_i}(2)\),
and so \(H^1\bigl(L_i, T_{\PP^r|L_i}(-p_i-q_i)\bigr) = 0\).

\vskip 3pt

It remains to show \(H^1\bigl(J_0, T_{\PP^r|J_0}(-\sum_{i=1}^{r+2} n_i)\bigr) = 0\).
Since \(\OO_{J_0}(n_1 + n_2 + \cdots + n_{r + 2})\in \mbox{Pic}^{r+2}(J_0)\) is general and \(\mu(T_{\PP^r|J_0}) = \frac{(r + 1)^2}{r} > r + 2\),
it suffices to show that \(T_{\PP^r|J_0}\)
satisfies the weak Raynaud condition.
This is a consequence of a special case (\(m = 0\)) of Proposition~\ref{prop-ss}.
\end{proof}

\begin{prop} \label{prop-e}
For each \(r + 1 \leq d \leq 2r - 1\) and \(0 \leq g \leq d + 1\), there exists a smooth elliptic curve \(J \subseteq  \PP^r\)
of degree \(d\),
meeting a rational normal curve \(R\) at a set \(\Gamma\) of \(g\) points,
for which \(T_{\PP^r|J}\) is a general vector bundle on \(J\).
\end{prop}
\begin{proof}
Applying Lemma~\ref{er-is-bn}, we may degenerate \(J \to \PP^r\)
to
\(f \colon J'=J_0 \cup L_1 \cup \cdots \cup L_{d-r-1} \to \PP^r\).
Applying Lemma~\ref{lm:general}, we thus reduce Proposition~\ref{prop-e}
to Proposition~\ref{prop-ss}.
\end{proof}

\begin{proof}[Proof of Theorem~\ref{thm:gen}.]
By Lemma~\ref{lm:d-r} (applied when $\epsilon=0$),
it suffices to consider the cases \(r + 1 \leq d \leq 2r - 1\).
This follows from Proposition~\ref{prop-e}.
\end{proof}

\section{The proof of the Minimal Resolution Property}

In this section, we prove Theorem~\ref{thm:main}. Via (\ref{eq:raynaud}) we have establish that this amounts to the statement
that if $C\subseteq \PP^r$ is a general Brill--Noether curve of genus $g$ and degree $d$ then the kernel bundle $M_V$ (or equivalently, the restricted tangent bundle $T_{\PP^r|C}=M_V^{\vee}\otimes L$) satisfies the strong Raynaud condition.
Our argument will use that certain vector bundles on elliptic curves
were shown to be general (and thus semistable) in the previous section.
This implies the strong Raynaud condition (in arbitrary characteristic):

\begin{lm} \label{ss-sr} ($\mathrm{char}(k)\geq 0$) A semistable vector bundle on an elliptic curve is strongly semistable and satisfies
the strong Raynaud condition.
\end{lm}
\begin{proof}
Since indecomposable vector bundles on an elliptic curve are semistable, via \cite[Theorem 2.16]{oda}
it follows that semistable vector bundles on an elliptic curve are strongly semistable.
Therefore, by  \cite[Corollary 7.3]{moriwaki}, the tensor product of semistable
vector bundles on an elliptic curve is semistable. Consequently, a wedge power of a semistable vector bundle on an elliptic curve is semistable.
Since semistable vector bundles on elliptic curves satisfy the weak Raynaud condition,
they therefore satisfy the strong Raynaud condition.
\end{proof}

\subsection{\boldmath The proof of Theorem \ref{thm:main}, respectively Theorem \ref{thm:semistab-strong}}\label{sec:thmmain}.

We fix $g, r$ and $d$ such that \(g \geq 1\), $d\geq 2r$, and $\rho(g,r,d)\geq 0$. We have to construct a BN curve $C\subseteq \PP^r$ of degree $d$ and genus $g$ for which $T_{\PP^r|C}$ satisfies the strong Raynaud condition, respectively is strongly semistable.
Using the identity $\rho(g,r,d)=\rho(g-r-1,r,d-r)$, by applying Lemma~\ref{lm:d-r}, we immediately reduce to the cases
\(2r + 1 \leq d \leq 3r - 1\)
(if \(d = 2r\) we reduce to the strong Raynaud condition, respectively strong semistability, for a rational normal curve).
%\ericcomment{Lemma~\ref{lm:d-r} is needed for all values of \(\epsilon\), not just \(\epsilon = r + 1\), since \(g - r - 1\) might be negative.}

Since \(d \leq 3r - 1\), we have
\[r(d - r + 1 - g) = \rho(g, r, d) + (3r - 1 - d) + 1 - r \geq 1 - r, \]
which implies $g \leq d-r+1$. The key input in these cases is Proposition~\ref{prop-e},
which implies that there exists a smooth elliptic curve \(J \subseteq  \PP^r\)
of degree \(d - r\),
meeting a rational normal curve \(R\) at \(g\leq d - r + 1\) points
\(p_1, p_2, \ldots, p_{g}\),
for which \(T_{\PP^r|J}\) is semistable
(and thus satisfies the strong Raynaud condition, respectively is strongly semistable, by Lemma~\ref{ss-sr}).
This implies that $H^1\bigl(J, T_{\PP^r|J}(-p_1 - \cdots - p_{g})\bigr) = 0$.

\vskip 3pt

Applying Lemma~\ref{lm:ray-open}, it suffices to show that the resulting curve
\[f \colon J \cup_{\{p_1, p_2, \ldots, p_g\}} R \hookrightarrow \PP^r\]
satisfies \(H^1(J\cup R, f^* T_{\PP^r}) = 0\) and is thus a BN curve.
For this, we use the exact sequence
\[0 \longrightarrow T_{\PP^r|J}(-p_1 - \cdots - p_g) \longrightarrow f^* T_{\PP^r} \longrightarrow  T_{\PP^r|R} \longrightarrow 0.\]
Since \(H^1\bigl(J, T_{\PP^r|J}(-p_1 - \cdots - p_{g})\bigr) = 0\), as well as \(H^1(R, T_{\PP^r|R})\cong H^1\bigl(\PP^1, \OO_{\PP^1}(r+1)\bigr)^{r}= 0\),
we conclude that \(H^1\bigl(J\cup R, f^* T_{\PP^r}\bigr) = 0\), as desired.
\hfill $\Box$

%\subsection{\boldmath The cases \(d < 2r\)}

%When \(d < 2r\), Theorem~\ref{thm:main} asserts that
%\(f^* T_{\PP^r}\) often does not satisfy the strong Raynaud condition.

%Here we show that the strong Raynaud condition fails whenever
%\(g \cdot (d - 2r) \leq -r\).
%Note that this condition implies
%\[(r + 1) \cdot [g \cdot (d - 2r) + r] - g \cdot \rho(d, g, r) \leq 0 \iff (g - 1)(g - r) \leq -1 \implies g < r.\]
%In particular, for general points \(p_1, p_2, \ldots, p_g \in C\),
%the map \(\bigoplus_{i = 1}^g T_{f \to p_i} \to f^* T_{\PP^r}\)
%is an injective map of sheaves.
%Since \(T_{f \to p_i} \simeq \O_C(1)(p_i)\), we may choose a nonvanishing section
%\(\sigma_i \in H^0(\O_C(1)(p_i))\). Then the wedge product defines a section
%\[\sigma_1 \wedge \cdots \wedge \sigma_g \in H^0(\wedge^g f^* T_{\PP^r} \otimes \O_C(-g)(-p_1 - \cdots - p_g)).\]
%Since \(\O_C(-g)(-p_1 - \cdots - p_g)\) is a general line bundle of degree \(-(d + 1)g\),
%this contradicts the strong Raynaud condition for \(f^* T_{\PP^r}\) if
%\[g \cdot \frac{(r + 1)d}{r} = \mu(\wedge^g f^* T_{\PP^r}) \leq (d + 1)g + (g - 1) \iff g \cdot (d - 2r) \leq -r.\]

\section{Butler's Conjecture}

In this section we prove Theorem~\ref{thm:stab} (Butler's conjecture) asserting that the kernel bundle $M_V$ (or equivalently, the restricted tangent bundle $T_{\PP^r|C}$) of a BN curve $C\subseteq \PP^r$ of genus $g\geq 3$ is stable, respectively semistable when $g=2$. We proceed by induction on \(g\).
Our argument will assume semistability for genus \(g - 1\) in order to derive \emph{stability} for genus \(g\);
we may therefore use \(g = 1\) as a base case, since semistability has already been
established in that case (cf.\ Theorem~\ref{thm:gen}). We assume \(g \geq 2\).

Applying Lemma~\ref{lm:d-r}, we reduce to the cases \(d \leq 2r\),
plus the case \((d, g) = (3r, 2)\).
The cases where \(d = 2r + 1\) reduce to \((d, g) = (r + 1, 1)\),
when \(T_{\PP^r|C}\) is stable by Theorem~\ref{thm:gen}.
The case \((d, g) = (3r, 2)\) would reduce to \((d, g) = (2r, 2)\) via Lemma~\ref{lm:d-r},
but Theorem~\ref{thm:stab} asserts that \(T_{\PP^r|C}\) is strictly semistable if \((d, g) = (2r, 2)\), and only in this case.

\subsection{\boldmath Corank $1$ subbundles of the kernel bundle for \(d \leq 2r\)}

We fix a BN curve $C\stackrel{|V|}\hookrightarrow \PP^r$ of genus $g$ and degree $d$ and set $L:=\OO_C(1)$. First we consider  corank $1$ subbundles \(F \subseteq T_{\PP^r|C}\).
We establish:

\begin{prop} \label{prop:c1q} Suppose that \(d \leq 2r\) and \(g \geq 2\) and \(r \geq 2\).
Then there exists a subbundle \(F \subseteq T_{\PP^r|C}\)
of corank \(1\) with \(\mu(F) \geq \mu(T_{\PP^r|C})\)
if and only if \(g = 2\) and \(d = 2r\).
Moreover, if \(g = 2\) and \(d = 2r\), then for any such subbundle, \(f^* T_{\PP^r} / F \cong \omega_C\otimes L\).
\end{prop}
\begin{proof}
Equivalently, we must show that there exists a quotient line bundle
\(T_{\PP^r|C} \twoheadrightarrow B\) with \(\mu(B) \leq \mu(f^* T_{\PP^r})\)
if and only if \(g = 2\) and \(d = 2r\), and that in this case, \(B \cong \omega_C\otimes L\).

The key point is that \(M_V^{\vee}\cong T_{\PP^r|C}(-1)\) is globally generated.
Therefore, if such a quotient line bundle \(B\) existed, then \(B\otimes L^{\vee}\) is a globally
generated \emph{line} bundle of degree at most \(\mu\bigl(M_V^{\vee}\bigr) =\frac{d}{r} \leq 2\).
No such line bundle exists on a general curve \(C\) of genus \(g \geq 3\),
and the only such line bundle on a curve of genus \(g = 2\) is  \(\omega_C\).
It remains only to observe that \(\Hom(f^* T_{\PP^r}, \omega_C\otimes L) \neq 0\), or equivalently
by Serre duality, that
\(H^1(C, M_V^{\vee}) \neq 0\). To see this, we observe
$h^0(C, M_{V}^{\vee}) \geq \mathrm{dim}(V) = r + 1 > r = \chi(C, M_V^{\vee})$.
\end{proof}

\subsection{\boldmath Main inductive argument: \(d \leq 2r\) and \((d, g) \neq (2r, r + 1)\)}

Here we give our main inductive argument, which applies when
\(d \leq 2r\) but \((d, g) \neq (2r, r + 1)\). Note that in this case $\rho(g-1,r,d-1)\geq 0$. We may start with a BN curve $C\subseteq \PP^r$ of degree \(d - 1\) and genus \(g - 1\) such that $T_{\PP^r|C}$ is semistable, and let
\(L_1\) be a \(2\)-secant line meeting \(C\) at general points \(p\) and \(q\).  We denote by
\(f \colon C \cup_{\{p, q\}} L_1 \to \PP^r\) the map inducing the corresponding embedding. Set $L:=\OO_C(1)\in W^{r}_{d-1}(C)$ and
$V:=f_{|C}^* H^0(\PP^r, \OO_{\PP^r}(1))\subseteq H^0(C,L)$.

\vskip 3pt

Write \(\nu \colon C \sqcup L_1 \to C \cup L_1\) for the normalization map,
and let \(F \subset \nu^* f^* T_{\PP^r} \cong T_{\PP^r|C} \oplus T_{\PP^r|L_1}\)
be any proper subbundle of  uniform rank $s$, with $1\leq s\leq r-2$.
Note that the case $s=r-1$ is covered by Proposition~\ref{prop:c1q}.
Our goal is to show that \(\mu^\adj(F) < \mu(f^* T_{\PP^r})\).

Write \(p_1\) and \(q_1\), respectively \(p_2\) and \(q_2\), for the points on \(L_1\), respectively \(C\), lying above \(p\) and \(q\).
Since \(T_{\PP^r|L_1} \cong \O_{L_1}(2) \oplus \O_{L_1}(1)^{r - 1}\), we have by inspection
\begin{equation} \label{muL-upper}
\mu(F_{|L_1}) \leq 1 + \frac{1}{s},
\end{equation}
therefore,
\begin{equation} \label{L-contrib}
\mu(F_{|L_1}) - \frac{1}{s} \cdot \Big[\codim_F \left(F_{|p_1} \cap F_{|p_2}\right) + \codim_F \left(F_{|q_1} \cap F_{|q_2} \right)\Big] \leq 1 + \frac{1}{s}.
\end{equation}
We now split into cases as follows.

\subsubsection{If inequality \eqref{L-contrib} is strict}
Since \(\mu\bigl(F_{|L_1}\bigr) \in \frac{1}{s} \cdot \ZZ\), we obtain
\[\mu(F_{|L_1}) - \frac{1}{s} \cdot \Big[\codim_F \left(F_{|p_1} \cap F_{|p_2}\right) + \codim_F \left(F_{|q_1} \cap F_{|q_2} \right)\Big] \leq 1 \implies \mu^\adj(F) \leq \mu(F_{|C}) + 1.\]
By induction, \(T_{\PP^r|C}\) is semistable.
Thus, as desired
\[\mu^\adj(F) \leq \mu(F_{|C}) + 1 \leq \mu(T_{\PP^r|C}) + 1 = \frac{(r + 1)(d - 1)}{r} + 1 < \frac{(r + 1)d}{r} = \mu(f^* T_{\PP^r}).\]

\subsubsection{If \eqref{L-contrib} is an equality}
We first rephrase both this condition, and our desired conclusion,
in terms of \(C\) alone. The desired conclusion  \(\mu^\adj(F) < \mu(f^* T_{\PP^r})\) is equivalent to
\begin{equation} \label{goal-L-eq}
\mu(F_{|C}) = \mu^\adj(F) - 1 - \frac{1}{s} < \mu(f^* T_{\PP^r}) - 1 - \frac{1}{s} = \frac{(r + 1)d}{r} - 1 - \frac{1}{s}.
\end{equation}

To rephrase the condition that \eqref{L-contrib} is an equality,
we first note that since $F_{|C}\subseteq T_{\PP^r|C}$, for any point \(x \in C\), there is a unique \(s\)-plane \(\Lambda_x \ni f(x)\)
with \(F_{| x} = T_{f(x)} \Lambda_x\).
Similarly, for any point \(y \in L_1\), we may associate an \(s\)-plane \(\Lambda_y^{L_1}\).
Since $F_{|L_1}\cong \OO_{L_1}(2)\oplus \OO_{L_1}^{s-1}$,  the $s$-plane   \(\Lambda_y^{L_1}\) is constant as \(y\in L_1\) varies.
The condition that \eqref{L-contrib} is an equality, translates into  \(F_{|p_1} = F_{|p_2}\), respectively \(F_{|q_1} = F_{|q_2}\).
We conclude that \(\Lambda_p = \Lambda_q\). Note that this can be rephrased in terms of a canonical identification of the fibres over the points $p$ and $q$
$$\Bigl(\frac{T_{\PP^r|C}}{F_{|C}}\Bigr)^{\vee}\otimes L_{|p}=\Bigl(\frac{T_{\PP^r|C}}{F_{|C}}\Bigr)^{\vee}\otimes  L_{|q} \subseteq V.$$

\vskip 3pt

\emph{Claim:}  The assumption \(\Lambda_p = \Lambda_q\) implies the inequality \eqref{goal-L-eq}.

\medskip

Let \(u\) be the maximal integer such that \(F_{|C}\) contains the subbundle \(T_{f_{|C} \to u\cdot p}\) of $T_{\PP^r|C}$. Similarly,
let \(v\) be the maximal integer such that \(F_{|C}\) contains \(T_{f|_C \to v\cdot q}\). In the language of kernel bundles, $u$ and $v$ are chosen maximally such that one has the inclusions $\Bigl(\frac{T_{\PP^r|C}}{F_{|C}}\Bigr)^{\vee}\otimes L\subseteq M_{V(-u\cdot p)}$ respectively
$\Bigl(\frac{T_{\PP^r|C}}{F_{|C}}\Bigr)^{\vee}\otimes L\subseteq M_{V(-v\cdot q)}$, where both $M_{V(-u\cdot p)}$ and $M_{V(-v\cdot q)}$ are regarded as subbundles of $M_V$.

\vskip 3pt

Write \(t = u + v\).
Note that by the definition of $u$ and $v$ we have
$$\Bigl(\frac{T_{\PP^r|C}}{F_{|C}}\Bigr)^{\vee}\otimes L_{|p}=\Bigl(\frac{T_{\PP^r|C}}{F_{|C}}\Bigr)^{\vee}\otimes L{_|q} \subseteq V\bigl(-(u+1)\cdot p\bigr)\cap V\bigl(-(v+1)\cdot q\bigr)=V\Bigl(-(u+1)\cdot p-(v+1)\cdot q\Bigr),$$

%$$\Lambda_p \supseteq \bigl\langle (u + 1)\cdot p\bigr\rangle=:\PP \Bigl(\frac{V}{V(-(u+1)\cdot p)}\Bigr)^{\vee} \ \ \mbox{ and } \ \ \Lambda_q %\supseteq \PP \bigl\langle(v + 1)\cdot q\bigr\rangle =:\Bigl(\frac{V}{V(-(v+1)\cdot q)}\Bigr)^{\vee},$$
%that is,  \(\Lambda_p = \Lambda_q\) contains \(\bigl\langle (u + 1)\cdot p + (v + 1)\cdot q\bigr\rangle\).
in particular, we obtain \(t = u + v \leq s - 1 \leq r - 3\).

\vskip 4pt

Let \(\bar{f} \colon C \to \PP^{r - t - 2}\) be the composition of \(f\) with the projection
from $\bigl\langle (u + 1)\cdot p + (v + 1)\cdot q\bigr\rangle$, and let \(K\) and \(I\) be the
kernel and the image of the quotient morphism
\begin{equation} \label{quot}
F_{|C} \longrightarrow  \bar{f}^* T_{\PP^{r - t - 2}}\bigl((u + 1)\cdot p + (v + 1)\cdot q\bigr).
\end{equation}
Let \(\bar{I}\) be the saturation of \(I\). By induction, we may assume
\(\bar{f}^* T_{\PP^{r - t - 2}}\) to be semistable, so
\begin{equation} \label{muI}
\mu(\bar{I}) \leq \mu(\bar{f}^* T_{\PP^{r - t - 2}}) + t + 2 = \frac{rd - dt - r - d  - 1}{r - t - 2}.
\end{equation}

By construction, \(K\) injects into \(T_{f|_C \to (u + 1)\cdot p} \oplus T_{f|_C \to (v + 1)\cdot q}\) and
contains \(T_{f|_C \to u\cdot p} \oplus T_{f|_C \to v\cdot q}\).
Write
\(K' \subseteq T_{f|_C \to (u + 1)\cdot p}/T_{f|_C \to u\cdot p} \oplus T_{f|_C \to (v + 1)\cdot q}/T_{f|_C \to v\cdot q} \cong L(p) \oplus L(q)\)
for the quotient.
By construction, \(K'\) is either zero or a line bundle not containing either of the factors.

\vskip 4pt

\noindent {\bf{(i)}} If \(K'\) is a line bundle, it does not contain either factor and since \(L(p) \not\cong L(q)\), we find \(\mbox{deg}(K')\leq d-1\). Therefore, using \(\mu(I) \leq \mu(\bar{I})\)
in combination with \eqref{muI}, we obtain
\begin{align*}
\mu(F_{|C}) &\leq \frac{t}{s} \cdot \mu\bigl(T_{f|_C \to u\cdot p} \oplus T_{f|_C \to v\cdot q}\bigr) + \frac{1}{s} \cdot \mu(K') + \frac{s - t - 1}{s} \cdot \mu(\bar{I}) \\
&\leq \frac{t}{s} \cdot d + \frac{1}{s} \cdot (d - 1) + \frac{s - t - 1}{s} \cdot \frac{rd - dt - r - d  - 1}{r - t - 2} \\
&= \frac{(r + 1)d}{r} - 1 - \frac{1}{s} - \frac{r(t + 1)(r - 1 - s) + (d - 2r)(rt - ts - 2s + r)}{sr(r - 2 - t)}.
\end{align*}
Comparing to \eqref{goal-L-eq}, all that remains is to verify \(P_1(d, r, s, t) > 0\),
where
\[P_1(d, r, s, t) = r(t + 1)(r - 1 - s) + (d - 2r)(rt - ts - 2s + r).\]
But evidently, \(P_1(2r, r, s, t) > 0\), and \(P_1(r, r, s, t) = r(s - 1 - t) \geq 0\).
Since \(r < d \leq 2r\), and \(P_1\) is linear in \(d\), we conclude that \(P_1(d, r, s, t) > 0\), as desired.

\vskip 4pt

\noindent {\bf{(ii)}} If \(K' = 0\), then the morphism \eqref{quot},
and therefore \(I \to \bar{I}\),
drops rank at both points \(p\) and \(q\).
Therefore, using inequality \eqref{muI},
\begin{align*}
\mu(F_{|C}) &\leq \frac{t}{s} \cdot \mu\bigl(T_{f|_C \to u\cdot p} \oplus T_{f|_C \to v\cdot q}\bigr) + \frac{s - t}{s} \cdot \left(\mu(\bar{I}) - \frac{2}{s - t}\right) \\
& \leq \frac{t}{s} \cdot d + \frac{s - t}{s} \cdot \left(\frac{rd - dt - r - d  - 1}{r - t - 2} - \frac{2}{s - t}\right) \\
&= \frac{(r + 1)d}{r} - 1 - \frac{1}{s} - \frac{r(t + 1)(r - 2 - s) + (d - 2r)(rt - ts - 2s)}{sr(r - 2 - t)}.
\end{align*}
Comparing to \eqref{goal-L-eq}, we are done if \(P_0(d, r, s, t) > 0\), where
\[P_0(d, r, s, t) = r(t + 1)(r - 2 - s) + (d - 2r)(rt - ts - 2s).\]
But evidently, \(P_0(2r, r, s, t) \geq 0\) with equality only if \(s = r - 2\). Moreover,
\[P_0(r, r, s, t) = r(r + s - 2 - 2t) \geq r(r + s - 2 - 2(s - 1)) = r(r - s) > 0.\]
Since \(r < d \leq 2r\), and \(P_0\) is linear in \(d\), we conclude that \(P_0(d, r, s, t) \geq 0\),
with equality only when \(d = 2r\) and \(s = r - 2\).

\vskip 4pt

It therefore remains only to show that it is impossible, when \(d = 2r\) and \(s = r - 2\),
to have equality everywhere in the above argument. Assume to the contrary that this occurs.
Then, $\deg \bigl(\frac{T_{\PP^r|C}}{F_{|C}}\Bigr)= 4r + 2$.
Moreover, the natural map \(L(p) \cong T_{f|_C \to (u + 1)\cdot p} / T_{f|_C \to u\cdot p} \to \frac{T_{\PP^r|C}}{F_{|C}}\) drops rank at \(q\),
so we obtain a map \(L(p + q) \to \frac{T_{\PP^r|C}}{F_{|C}}\). Exchanging the roles
of \(p\) and \(q\), we obtain another such map.
Combining these, we obtain a map
\begin{equation} \label{quot-iso}
L(p + q)^{\oplus 2} \longrightarrow \frac{T_{\PP^r|C}}{F_{|C}}.
\end{equation}
Away from \(\{p, q\}\), the kernel of \eqref{quot-iso} coincides with \(K'\) by definition. Since \(K' = 0\) in case (ii) by assumption, and the kernel of \eqref{quot-iso} is a priori torsion-free, \eqref{quot-iso} is necessarily an injection.
Recalling that $\mbox{deg}(L)=d-1=2r-1$, we observe that both sides of \eqref{quot-iso} are rank \(2\) vector bundles of degree \(4r + 2\), hence  \eqref{quot-iso} is an isomorphism.
In particular,
\[\dim \Hom\bigl(T_{\PP^r|C}, L(p + q)\bigr) \geq 2\Leftrightarrow h^0\bigl(C, M_V(p+q)\bigr)\geq 2.\]
Consequently, for general points \(x_1, x_2, \ldots, x_{g - 2} \in C\), setting $\eta:=\OO_C(p+q+x_1+\cdots+x_{g-2})$ we have
$h^0\bigl(C, M_V\otimes \eta)\geq 2$. Since $\chi\bigl(C, M_V\otimes \eta)<0$ and the line bundle $\eta\in \mbox{Pic}^g(C)$ is general, this contradicts that $M_V$ satisfies the weak Raynaud condition
(cf.\ Proposition~\ref{prop:weak}), providing the desired contradiction.

\subsection{\boldmath The case of canonical curves \((d, g) = (2r, r + 1)\)\label{sec:canonical}}

In this case $C\subseteq \PP^r$ is a general canonical curve and the stability of the restricted tangent bundle $T_{\PP^r|C}$ is well known for \emph{every} non-hyperelliptic smooth curve, see \cite{PR}, \cite{EL}.  However, this case will also be treated in Section \ref{sec:frobenius} with our methods --- and we obtain not just stability, but \emph{strong} stability of $M_{\omega_C}$, something the methods of \cite{PR}, \cite{EL} do not seem to lead to.

\subsection{\boldmath The case \((d, g) = (3r, 2)\)}

Finally, we consider the case \(g = 2\) and \(d = 3r\).
In this case, we prove the semistability of \(T_{\PP^r|C}\) by induction on \(r\).
The base case \(r = 1\) being trivial, we suppose for the inductive step that \(r \geq 2\).

\vskip 3pt

We degenerate to a reducible curve
\(f\colon C \cup_{p_1} L_1 \cup_{p_2} L_2 \to \PP^r\), where $C\subseteq \PP^r$ is a BN curve
of genus $2$ and degree \(3r- 2\)  and \(L_1, L_2\) are \(1\)-secant lines
meeting \(C\) at \(p_1, p_2\) respectively. Write $L:=\OO_C(1)\in W^r_{3r-2}(C)$ and $V:=f_{|C}^* H^0(\PP^r, \OO_{\PP^r}(1))$. Choose furthermore points \(q_i \in L_i \setminus \{p_i\}\).
Applying Lemma~\ref{lm:stab-mod}, it suffices to show the stability of
\(T_{\PP^r|C} [p_1 \to q_1][p_2 \to q_2]\). After further specializing \(q_1\) and \(q_2\)
to a common general point \(q\in \PP^r\), it suffices to establish the stability of
\(T_{\PP^r|C} [p_1 + p_2 \to q]\). We apply Lemma~\ref{rat-base-stab} to this family of vector bundles,
having the base \(B = \PP^r \setminus \{p_1, p_2\}\) parameterizing the position of \(q\).
If \(q \in C\), we have the exact sequence
\begin{equation} \label{seq-for-2}
0 \longrightarrow L(q) \cong T_{f|_C \to q} \longrightarrow T_{\PP^r|C} [p_1 + p_2 \to q] \longrightarrow  M_{V(-q)}^{\vee}\otimes L(-p_1-p_2) \longrightarrow  0.
\end{equation}
As desired, \(\mbox{deg}\bigl(L(q)\bigr)=3r - 1 = \mu\bigl(M_{V(-q)}^{\vee}\otimes L(-p_1-p_2)\bigr)\),
and \(M_{V(-q)}^{\vee}\otimes L(-p_1-p_2)\) is stable
by induction, which completes the proof unless \(r = 2\). In this case, we must also show that
\(L(q)\) \(\ncong M_{V(-q)}\otimes L(-p_1-p_2)\), which is immediate
and that the extension \eqref{seq-for-2} is nonsplit. To that end, it suffices to check that $\Hom\bigl(L^{2}(-q-p_1-p_2), T_{\PP^r|C} [p_1 + p_2 \to q]\bigr) = 0$,
which holds since \(T_{\PP^r|C}\) satisfies the weak Raynaud condition and
\(L^{-2}(q + p_1 + p_2)\in \mbox{Pic}^{-5}(C)\) is  general.
\hfill $\Box$

\section{Strong stability of kernel bundles in positive characteristic.}\label{sec:frobenius}
 We now work over an algebraically closed field $k$ of characteristic $p>0$ and  finally prove Theorem \ref{thm:stab-strong}.
Fix $g, r$ and $d$ satisfying one of the three cases of Theorem~\ref{thm:stab-strong}.
We seek to construct a general Brill--Noether curve $C\subseteq \PP^r$ of degree $d$ and genus $g$ for which $T_{\PP^r|C}$ is strongly stable.

\vskip 4pt

\noindent
\emph{Case~\eqref{ss1}:} Applying Lemma~\ref{lm:d-r},  we reduce to the case \(d = 2r, g=r+1\) of canonical curves.

\vskip 4pt

We degenerate to a map from a reducible curve \(f \colon R' \cup_\Gamma R'' \to \PP^r\),
where $R'$ and $R''$ are smooth rational curves meeting a set
\(\Gamma = \{p_1, p_2, \ldots, p_{r + 2}\}\) of \(r + 2\) general points.
We consider an iterate $F^e$ of the Frobenius morphism, and write \(\nu \colon R' \sqcup R'' \to R' \cup R''\) for the normalization map.
Consider any subbundle of uniform rank \(\F \subseteq (F^e)^* \nu^*T_{\PP^r|R'\cup R''} \cong (F^e)^*\bigl(T_{\PP^r|R'}\bigr) \oplus
(F^e)^*\bigl(T_{\PP^r|R''}\bigr)\).
Our goal is to show that \(\mu^\adj(\F) < \mu\bigl((F^e)^*T_{\PP^r|R'\cup R''}\bigr)\).

Since \(T_{\PP^r|R'}\) is perfectly balanced,
the corresponding projective bundle
\(\PP T_{\PP^r|R'}\) is trivial.
Note that \(\mu(\F_{|R'}) \leq \mu\bigl((F^e)^* T_{\PP^r|R'}\bigr)\),
with equality if and only if \(\F_{|R'}\) is perfectly balanced. A similar statement holds for  \(\F_{R''}\).
Thus \(\mu(\F) \leq \mu\bigl((F^e)^*T_{\PP^r|R'\cup R''}\bigr)\),
with equality only when both \(\F_{|R'}\) and \(\F_{|R''}\) are perfectly balanced.
Write \(p_i'\), respectively \(p_i''\), for the points lying above \(p_i\)
on \(R'\), respectively on \(R''\).
Then \(\mu^\adj(\F) \leq \mu(\F)\), with equality only when
\(\F_{|p_i'} = \F_{|p_i''}\) for \(i=1, \ldots, r+2\).
Putting these together, we have \(\mu^\adj(\F) \leq \mu\bigl((F^e)^* T_{\PP^r|R'\cup R''}\bigr)\),
with equality only if \(\F_{|R'}\) and \(\F_{|R''}\) are both perfectly balanced
and \(\F_{|p_i'} = \F_{|p_i''}\) for every \(i\).

\vskip 4pt

To express this in a more convenient way, the triviality of the bundles \(\PP T_{\PP^r|R'}\)
and \(\PP T_{\PP^r|R''}\) give rise to two sets
of canonical isomorphisms:
\[\phi_{ij}' \colon \PP T_{p_i}(\PP^r) \cong \PP T_{p_j}(\PP^r) \andlarge \phi_{ij}'' \colon \PP T_{p_i}(\PP^r) \cong \PP T_{p_j}(\PP^r).\]
If such a subbundle \(\F\) exists, then \(\Lambda_i = \PP \F_{|p_i'} = \PP \F_{|p_i''}\)
would define a collection of proper subspaces
\(\Lambda_i \subseteq \PP T_{p_i}(\PP^r)\)
carried into each other by both the \(\phi_{ij}'\) and the \(\phi_{ij}''\).
Our task is to show that no such collection of proper subspaces \(\Lambda_i\) exist, as long the rational normal curves $R', R''$ and the points $p_1, \ldots, p_{r+2}$ are chosen generically.

\vskip 3pt

For a point \(x\in R\),
we have \(T_{R \to x} \cong \O_{\PP^1}(r + 1)\);
thus \(\phi_{ij}'\) carries \(\PP T_{R' \to x|p_i}\) to \(\PP T_{R' \to x|p_j}\).
Therefore, denoting by $\overrightarrow{p_i, p_j}\in \PP T_{p_i}(\PP^r)$ the tangent direction determined by $\langle p_i, p_j\rangle$, we obtain
\begin{equation} \label{tran}
\phi_{12}'(\overrightarrow{p_1, p_j}) = \begin{cases}
\overrightarrow{p_2, p_j} & \text{if \(j = 3, \ldots, r + 2\),} \\
\PP T_{p_2} R' & \text{if \(j = 2\).}
\end{cases}
\end{equation}

In particular, we see that \(\overrightarrow{p_1, p_j}\) are  eigenspaces for
\(\phi_{21}' \circ \phi_{12}''\) for $j=3, \ldots, r+2$, hence \(\phi_{21}' \circ \phi_{12}''\) is diagonalizable.
Note that the space of reducible curves \(f \colon R' \cup R'' \to \PP^r\) equipped with an \emph{ordering} of the marked points  \(p_i\)
is irreducible, so by monodromy considerations, if \emph{some} two eigenvalues of \(\phi_{21}' \circ \phi_{12}''\)
are equal, then \emph{any} two eigenvalues are equal.
Therefore, \(\phi_{21}' \circ \phi_{12}''\)
must either be be the identity or have distinct eigenvalues.
However, the former case is ruled out by applying \eqref{tran} for \(j = 2\),
since \(T_{p_2} R' \neq T_{p_2} R''\).
Thus \(\phi_{21}' \circ \phi_{12}''\) is diagonalizable with distinct eigenvalues.

In particular, for any collection of proper subspaces \(\Lambda_i \subseteq \PP T_{p_i}(\PP^r)\)
carried into each other by both the \(\phi_{ij}'\) and the \(\phi_{ij}''\),
the subspace \(\Lambda_1\) must be a span of eigenvectors of \(\phi_{21}' \circ \phi_{12}''\), that is,

$$\Lambda_1 = \PP \Bigl(T_{p_1} \langle p_1, p_{i_1}, \ldots, p_{i_s} \rangle\Bigr) \subseteq \PP T_{p_1}(\PP^r), \  \mbox{ for } i_1, \ldots, i_s \neq 2$$
Since $\mbox{rk}(\F)<r$,
such a representation is unique.
% (because the projections of \(p_2, p_3, \ldots, p_{r + 2}\) from \(p_1\)
%are in linear general position in \(\PP^{r - 1}\)).
Since the ordering of the points was arbitrary,
we must therefore have \(i_1, \ldots, i_s \neq n\) for every \(n = 2, \ldots, r + 2\),
which is impossible. Therefore no such collection of subspaces \(\Lambda_i\) exists,
as desired.

\vskip 4pt

\noindent
\emph{Case~\eqref{ss2}:} We begin by applying Lemma~\ref{lm:d-r} to reduce to the cases
\(2r + 1 \leq d \leq 3r\). If \(d \leq 3r - 1\), then as in Section~\ref{sec:thmmain}, we have \(g \leq d - r + 1\).
But if \(d = 3r\), then we may further assume \(g \leq r \leq d - r + 1\),
since otherwise we fall into case~\eqref{ss1}.
Therefore we may suppose \(g \leq d - r + 1\) regardless.

\vskip 3pt

By Proposition~\ref{prop-e},
there is a smooth elliptic curve \(J \subseteq  \PP^r\)
of degree \(d - r - 1\),
meeting a rational normal curve \(R\) at \(g - 1 \leq d - r\) points
\(p_1, p_2, \ldots, p_{g - 1}\),
for which \(T_{\PP^r|J}\) is semistable and thus strongly semistable.
This implies that $H^1\bigl(J, T_{\PP^r|J}(-p_1 - \cdots - p_{g - 1})\bigr) = 0$.

Let \(L_1\) be a \(2\)-secant line to \(R\).
Then \(R \cup L_1\) can be smoothed to a general elliptic curve \(J'\) of degree \(r + 1\).
Since $H^1\bigl(J, T_{\PP^r|J}(-p_1 - \cdots - p_{g - 1})\bigr) = 0$,
we may lift this deformation to a deformation of \(J\) that
continues to meet \(J'\) at \(g - 1\) points.
Applying Proposition~\ref{prop-e} again, \(T_{\PP^r|J'}\) is a general vector bundle on \(J'\),
and therefore strongly semistable. Since both its degree and the characteristic are prime to its
rank, \(T_{\PP^r|J'}\) must be strongly stable.
Therefore, applying Lemma~\ref{lm:components-to-special-strong},
it suffices to show that the resulting stable curve of genus $g$ and degree $d$
\[f \colon J \cup_{\{p_1, p_2, \ldots, p_{g-1}\}} J' \hookrightarrow \PP^r\]
satisfies \(H^1(J\cup J', f^* T_{\PP^r}) = 0\) and is thus a BN curve.
For this, we use the exact sequence
\[0 \longrightarrow T_{\PP^r|J}(-p_1 - \cdots - p_{g-1}) \longrightarrow f^* T_{\PP^r} \longrightarrow  T_{\PP^r|J'} \longrightarrow 0.\]
Since \(H^1\bigl(J, T_{\PP^r|J}(-p_1 - \cdots - p_{g - 1})\bigr) = H^1(J', T_{\PP^r|J'}) = 0\),
we have \(H^1\bigl(J\cup J', f^* T_{\PP^r}\bigr) = 0\), as desired.

\vskip 4pt

\noindent
\emph{Case~\eqref{ss3}:} We begin by applying Lemma~\ref{lm:d-r} to reduce to the cases
\(3r \leq d \leq 4r - 1\). As in Section~\ref{sec:thmmain}, we have \(g \leq d - r + 2\).

By Proposition~\ref{prop-e},
there is a smooth elliptic curve \(J \subseteq  \PP^r\)
of degree \(d - 2r\),
meeting a rational normal curve \(R\) at \(g - r - 1 \leq d - 2r + 1\) points
\(p_1, p_2, \ldots, p_{g - r - 1}\),
for which \(T_{\PP^r|J}\) is semistable and thus strongly semistable.
This implies that $H^1\bigl(J, T_{\PP^r|J}(-p_1 - \cdots - p_{g - r - 1})\bigr) = 0$.

Let \(R'\) be another rational normal curve, meeting \(R\) in exactly \(r + 2\) points.
As in the previous case,
we may smooth \(R \cup R'\) to a general canonical curve \(C'\),
while deforming \(J\) so it continues to meet \(J'\) at \(g - r - 1\) points.
From Section~\ref{sec:canonical},
we know \(T_{\PP^r|C'}\) is strongly stable.
Therefore, applying Lemma~\ref{lm:components-to-special-strong},
we can complete the proof by arguing, as in the previous case, that the resulting curve
\[f \colon J \cup_{\{p_1, p_2, \ldots, p_{g-r-1}\}} C' \hookrightarrow \PP^r\]
satisfies \(H^1(J\cup C', f^* T_{\PP^r}) = 0\) and is thus a BN curve.
\hfill $\Box$

\begin{rem}\label{rmk:fermat}
For those values of $g$ and $d$ appearing in the statement of Theorem \ref{thm:semistab-strong}, its conclusions are optimal even in the case of canonical curves. For instance, in genus $3$, that is, for smooth quartics $C\subseteq \PP^2$, the kernel bundle $M_{\omega_C}$ is always semistable  \cite[Corollary 3.5]{Tr}. However, for the Fermat curve $C:(x^4+y^4+z^4=0)$ putting together results of Han--Monsky \cite{HM} and \cite{Tr2},  for characteristic $p\geq 17$, when $p\equiv \pm 1 \mbox{ mod } 8$, then $M_{\omega_C}$ is strongly semistable (thus $e_{\mathrm{HK}}(C)=3$), whereas if $p\equiv \pm 3 \mbox{ mod } 8$, then $F^*(M_{\omega_C})$ is not semistable, in which case $e_{\mathrm{HK}}(C)=3+\frac{1}{p^2}$. Similar results exist for the Klein quartic $C:(x^3y+y^3z+z^3x=0)$, see \cite{Tr2}. For a characteristic $p\geq 17$, when $p\equiv \pm 1 \mbox{ mod } 7$, then $M_{\omega_C}$ is strongly semistable, whereas for $p \equiv \pm 2 \mbox{ mod } 7$, the second Frobenius pullback $(F^2)^*\bigl(M_{\omega_C}\bigr)$ is not semistable. We expect such phenomena to propagate throughout when studying individual canonical curves of higher genus.
\end{rem}

\bibliographystyle{plain}

\end{document}